\pdfoutput=1
\documentclass{article}

\def\preprint{0}
\usepackage{ifthen}

\usepackage{fullpage}
\usepackage{amsthm}
\usepackage{amsmath,amsfonts,amssymb,multirow}
\usepackage{algorithm} 
\usepackage{algpseudocode} 
\usepackage{graphicx} 
\usepackage{csquotes} 
\usepackage{placeins} 
\usepackage[normalem]{ulem} 
\usepackage{xcolor}
\usepackage{hyperref}
\usepackage[capitalise,noabbrev]{cleveref} 
\usepackage{subcaption}
\usepackage[noblocks]{authblk} 
\usepackage{marvosym} 
\usepackage[square, numbers]{natbib}
\usepackage{booktabs}

\usepackage{tikz}
\usetikzlibrary{arrows,calc}
\usepackage{url}
\usepackage{cite}
\usepackage{color}
\usepackage{epsfig}
\usepackage{epstopdf}
\usepackage{paralist} 
\usepackage{bbm} 
\usepackage{wrapfig} 
\usepackage{comment} 
\usepackage{palatino} 

\usepackage{aligned-overset} 

\providecommand{\oline}[1]{\mkern 1.5mu\overline{\mkern-1.5mu#1}}

\renewcommand{\hbar}{\oline{h}}

\newcommand{\Embb}{\mathbb{E}}

\newcommand{\Rmbb}{\mathbb{R}}

\newcommand{\vb}{\mathbf{b}}
\newcommand{\vc}{\mathbf{c}}

\newcommand{\vx}{\mathbf{x}}
\newcommand{\vy}{\mathbf{y}}
\newcommand{\vz}{\mathbf{z}}
\newcommand{\mA}{\mathbf{A}}
\newcommand{\mB}{\mathbf{B}}

\newcommand{\mD}{\mathbf{D}}

\newcommand{\mI}{\mathbf{I}}

\newcommand{\mM}{\mathbf{M}}

\newcommand{\mP}{\mathbf{P}}

\newcommand{\mS}{\mathbf{S}}

\newcommand{\mU}{\mathbf{U}}
\newcommand{\mV}{\mathbf{V}}
\newcommand{\mW}{\mathbf{W}}

\newcommand{\mZ}{\mathbf{Z}}

\newcommand{\cD}{{\cal D}}

\newcommand{\parens}[1]{\left( #1 \right)}
\newcommand{\bracks}[1]{\left[ #1 \right]}
\newcommand{\norm}[1]{\left\| #1 \right\|}
\newcommand{\normsq}[1]{\norm{#1}^2}
\newcommand{\abs}[1]{\left| #1 \right|}

\newcommand{\qtext}[1]{\quad \text{#1} \quad}

\newcommand{\trans}{{T}}
\newcommand{\inv}{{-1}}
\newcommand{\pinv}{{\dagger}}

\newcommand{\eqdef}{\overset{\mathop{def}}{=}}
\newcommand{\thup}{^{\text{th}}}

\newcommand{\E}{\Embb}

\newcommand{\R}{\Rmbb}

\newcommand{\Eb}[1]{\E\bracks{#1}}

\DeclareMathOperator*{\argmin}   {arg\,min}

\DeclareMathOperator*{\proj}     {proj}

\DeclareMathOperator{\rowspace} {row}
\DeclareMathOperator{\colspace} {col}

\newtheorem{theorem}{Theorem}[section]

\crefname{assumption}{Assumption}{Assumptions}

\newtheorem{remark}[theorem]{Remark}
\newtheorem{prop}[theorem]{Proposition}

\newcommand{\vecstack}[2]{\begin{bmatrix} #1 \\ #2 \end{bmatrix}}
\newcommand{\colj}{{:j}}
\newcommand{\rowi}{{i:}}

\newcommand{\subij}{{ij}}

\newcommand{\bi}{b_i}
\newcommand{\Ai}{\mA_{\rowi}}
\newcommand{\Aj}{\mA_{\colj}}
\newcommand{\Aij}{A_{ij}}

\newcommand{\yinit}{\vy^0}
\newcommand{\yone}{\vy^1}
\newcommand{\ytwo}{\vy^2}
\newcommand{\yk}{\vy^k}
\newcommand{\ykpo}{\vy^{k+1}}
\newcommand{\yopt}{\vy^\star}

\newcommand{\xinit}{\vx^0}

\newcommand{\xk}{\vx^k}
\newcommand{\xkpo}{\vx^{k+1}}
\newcommand{\xopt}{\vx^\star}

\newcommand{\zinit}{\vz^0}

\newcommand{\zk}{\vz^k}
\newcommand{\zkpo}{\vz^{k+1}}
\newcommand{\zopt}{\vz^\star}

\newcommand{\zkpoi}{z^{k+1}_i}

\newcommand{\zx}{\vecstack{\vz}{\vx}}

\newcommand{\zxk}{\vecstack{\zk}{\xk}}
\newcommand{\zxkpo}{\vecstack{\zkpo}{\xkpo}}

\newcommand{\rhsi}{\vecstack{0}{b_i}}
    
\newcommand{\eigmin}{\lambda_{\min}^+}

\newcommand{\prows}{\mP_{\text{rows}}}
\newcommand{\pcols}{\mP_{\text{cols}}}
\newcommand{\drows}{\mD_{\text{rows}}}
\newcommand{\dcols}{\mD_{\text{cols}}}

\ifthenelse{\preprint = 0}{
    \usepackage[colorinlistoftodos,bordercolor=orange,backgroundcolor=orange!20,linecolor=orange,textsize=scriptsize]{todonotes}
    \newcommand{\jdm}[1]{\todo[inline]{\textbf{Jacob: }#1}}
    \newcommand{\ben}[1]{\todo[inline]{\textbf{Ben: }#1}}
    \newcommand{\nate}[1]{\todo[inline]{\textbf{Nate: }#1}}
}{
    \newcommand{\jdm}[1]{}
    \newcommand{\ben}[1]{}
    \newcommand{\nate}[1]{}
}

\graphicspath{{./figures/}{}} 

\author{Benjamin Jarman}
\affil{University of California, Los Angeles, Los Angeles, CA}
\affil{\href{mailto:bjarman@math.ucla.edu}{bjarman@math.ucla.edu}}

\author{Nathan Mankovich}
\affil{Colorado State University, Fort Collins, CO}
\affil{\href{mailto:nate.mankovich@colostate.edu}{nate.mankovich@colostate.edu}}

\author{Jacob D. Moorman}
\affil{University of California, Los Angeles, Los Angeles, CA}
\affil{\href{mailto:jacob@moorman.me}{jacob@moorman.me}}

\title{Randomized Extended Kaczmarz is a Limit Point of Sketch-and-Project}
\date{}

\begin{document}

\maketitle
\begin{abstract}
    The sketch-and-project (SAP) framework for solving systems of linear equations has unified the theory behind popular projective iterative methods such as randomized Kaczmarz, randomized coordinate descent, and variants thereof. The randomized extended Kaczmarz (REK) method is a popular extension of randomized Kaczmarz for solving inconsistent systems, which has not yet been shown to lie within the SAP framework. In this work we show that, in a certain sense, REK may be expressed as the limit point of a family of SAP methods, but we argue that it is unlikely that REK can be translated into a SAP method itself. We provide an extensive theoretical analysis of the family of methods comprising said limit, including convergence guarantees and further connections to REK. We follow this with an array of experiments demonstrating these methods and their connections in practice.
\end{abstract}

\section{Introduction}


Solving systems of linear equations of the form $\mA \vx = \vb$ is a foundational problem in applied mathematics. It arises in contexts from image and signal processing \citep{natterer, Feichtinger1992NewVO, Herman1993AlgebraicRT}, to machine learning subroutines and modern data science \citep{leskovec_rajaraman_ullman_2014}. 
Classical methods such as Gaussian elimination or Cholesky decomposition do not scale well to these big-data-driven applications due to their high computational cost and the requirement of loading large portions of the system into memory. 
In the search for fast and scalable solvers, projection-based iterative methods have risen in popularity \citep{popa2012projection}, due to their lower computational cost and low memory requirements.

One such method is due to Kaczmarz \citep{ogkacz}, rediscovered in the computer tomography literature as the Algebraic Reconstruction Technique \citep{hounsfield_CAT, Herman1993AlgebraicRT}:
an initial guess $\xinit$ is iteratively projected onto the hyperplanes defined by the rows of the system, $\{\vx : \Ai \vx = \bi\}$ for row indices $i$. 
In the original Kaczmarz method, the hyperplanes were cycled through in order \citep{ogkacz}, however, this can produce slow convergence when consecutive hyperplanes have small incident angle.
\citet{vershstrohkacz} showed that a randomized variant of the method - randomized Kaczmarz (RK) - does not suffer from the possible unlucky ordering of the equations, and will in fact achieve linear convergence in expectation.
At each iteration a row is chosen with probability proportional to its Euclidean norm: applied to the system $\mA \vx = \vb$, the Kaczmarz update is
\begin{equation}\label{eqn:rk-update}
    \xkpo = \xk - \frac{\Ai\xk - \bi}{\norm{\Ai}^2}\Ai^\trans,
\end{equation}
where $i$ is the index of the equation chosen at the $k\thup$ iteration, and $\mA_\rowi$ is the $i\thup$ row of $\mA$. 
If the system has a least-norm solution $\xopt$ and the initial guess is chosen in the row-space of the matrix $\xinit \in \rowspace(\mA)$, RK enjoys linear convergence in the expectation of the mean squared error: initially proven in \citep{vershstrohkacz} with a convergence rate depending on the smallest singular value of $\mA$, the result was improved in \citep{rek} to depend on the smallest \textit{positive} singular value.
\begin{theorem}[\citet{vershstrohkacz,rek}]\label{rkconvergence}
    Let $\mA \in \mathbb{R}^{m \times n}$, $\vb \in \mathbb{R}^m$ be such that the system $\mA \vx = \vb$ is consistent, with solution $\xopt$. Then the iterates of RK applied to this system satisfy
    \begin{equation}
        \Eb{\norm{\xk - \xopt}^2} \leq \left(1 - \frac{\sigma_{\min}^2(\mA)}{\norm{\mA}_F^2}\right)^k\norm{\xinit - \xopt}^2,
    \end{equation}
    where $\sigma_{\min}(\mA)$ is the smallest nonzero singular value of the matrix, $\norm{\cdot}$ is the Euclidean norm, and $\norm{\cdot}_F$ is the Fr\"obenius norm.
\end{theorem}


Motivated by RK, a multitude of row-projection type methods have been developed: block variants, where one projects onto an entire block of rows at each iteration \citep{needell2014paved}; averaged variants, where one takes a sample of rows and averages their projections \citep{moorman_randomized_2020,necoara_faster_2019}; greedy methods, where one projects onto rows producing large movements \citep{Haddock2021GreedWA,gower_adaptive_2019,bai_greedy_2018}, and more \citep{needell2013two}. These methods each achieve linear convergence in mean squared error for consistent systems analogous to that of RK with convergence rates depending on properties of the matrix $\mA$ as well as parameters such as pavings for block methods and sample size for averaged methods.

If the system $\mA \vx = \vb$ is not consistent and we seek the least squares solution $\xopt = \mA^\dagger \vb$, RK is only guaranteed to converge up to a horizon, past which convergence is no longer guaranteed.

\begin{theorem}[\citet{needellnoisyrk}]
Let $\mA \in \mathbb{R}^{m \times n}$, $\vb \in \mathbb{R}^m$, and let $\xopt$ be the least squares solution of the system $\mA \vx = \vb$. Then the iterates of RK applied to this system satisfy:
\begin{equation*}
    \Eb{\norm{\xk - \xopt}^2} \leq \left(1 - \frac{\sigma_{\min}^2(\mA)}{\norm{\mA}_F^2}\right)^k \norm{\xinit - \xopt}^2 + \max_i \frac{ \abs{\mA_\rowi \xopt - b_i}}{\norm{\mA_\rowi}^2}
\end{equation*}
\end{theorem}
Similar results exist, for both consistent and inconsistent systems, for block and greedy variants of RK--see e.g. \citep{needell2014paved, Haddock2021GreedWA}.

To eliminate this convergence horizon and achieve convergence all the way to $\xopt$, \citet{rek} proposed randomized extended Kaczmarz (REK). REK uses an auxiliary iterate $\zk$ to approximate $\zopt=\proj_{\ker(\mA^\trans)}(\vb)$ so that $\mA \vx = \vb - \zk$ approaches the consistent system $\mA \vx = \vb - \zopt$, which is solved by iterations of RK.
We give a more detailed introduction in \cref{rekbackground}. 
Similar extensions have also been found for block methods, see e.g. \citep{Du2020RandomizedEA}.

RK, its variants, and other iterative projective linear system solvers share a common theme: at each iteration, a sample of rows or columns of $\mA$---a sketch---is taken, and the previous iterate is projected onto the solution space of the sketched system. Gower and Richt\'arik \citep{gower_randomized_2015} introduced a framework to unify such methods, named Sketch-and-Project (SAP). They introduced a general method depending on a distribution of sketching matrices $\mathcal{D}$ and a symmetric positive definite matrix $\mB$, where at each iteration a sketching matrix $\mS$ is drawn from $\mathcal{D}$ and the previous iterate is projected onto the solution space of $\mS^\trans \mA \vx = \mS^\trans \vb$, relative to the norm induced by $\mB$. They showed that particular choices of $\mathcal{D}$ and $\mB$ give rise to RK, randomized coordinate descent, and other popular solvers. Furthermore, they proved a general convergence result for SAP from which best-known convergence results of said methods may be recovered.
We give a more detailed introduction to SAP in \cref{sapbackground}.

Up to this point, it was not known whether REK can be formulated as a SAP method, despite its similarity to RK and methodology of solving a sketched system. 
In this paper, we demonstrate that for a suitably embedded problem, it is unlikely that REK may be expressed as a SAP method. However, we show that REK may instead be extracted as a limit point (in a particular sense) of a family of SAP methods.  
We study the convergence of the methods that lie along the limit both theoretically and empirically.

\subsection{Organization}


The rest of the paper is organized as follows. In the remainder of the introduction we introduce our notation (\ref{notation}) and our main contributions (\ref{contributions}). In \cref{background} we provide detailed backgrounds for REK (\ref{rekbackground}) and SAP (\ref{sapbackground}), including pseudocode and known convergence results. In \cref{bridging} we discuss directions suggesting that REK is not a member of the SAP family (\ref{rekisnotsap}), and present a specific SAP method SAP-REK$(\epsilon)$ depending on a parameter $\epsilon$, whereupon taking the limit $\epsilon \to 0$ we recover the REK update (\ref{SAP-REKlimit}). In  \cref{SAP-REKanalysis} we provide a detailed theoretical analysis of SAP-REK$(\epsilon)$ and its behavior in said limit, and in \cref{experiments} we give an analogous empirical analysis. Lastly, in \cref{conclusion}, we conclude and discuss future directions.
\subsection{Notation}
\label{notation}
Throughout, we indicate matrix and vector quantities $\mA, \mI, \vx, \vb, \ldots$ with boldface, and scalars $\epsilon, m, b_i, \ldots$ in regular type. The transpose and pseudoinverse of a matrix $\mA$ are denoted $\mA^T$ and $\mA^\dagger$ respectively. When working with a system of equations $\mA \vx = \vb$, we use $m$ and $n$ to represent the number of rows and columns respectively so that $\mA \in \mathbb{R}^{m \times n}$ and $\vb \in \mathbb{R}^{m}$. For indexes $i \in \{1, \ldots, m\}$ and $j \in \{1, \ldots, n\}$, we use $\Ai$ and $\Aj$ to refer to the $i\thup$ row and $j\thup$ column of a matrix $\mA$. The row and column spaces of a matrix $\mA$ are denoted $\rowspace(\mA)$ and $\colspace(\mA)$.


The nonzero singular values of a matrix $\mA$ are $\sigma_1(\mA) \geq \sigma_2(\mA) \geq \ldots \sigma_{r}(\mA):= \sigma_{\min}(\mA) > 0$, where $r$ is the rank of $\mA$. 
For a symmetric positive semi-definite matrix $\mW$, the smallest positive eigenvalue is denoted $\lambda_{\min}^+(\mW)$. 
The Fr\"obenius norm of a matrix $\mA$ is denoted $\norm{\mA}_F = \sqrt{\sum_{i=1}^{m}\sum_{j=1}^{n}\Aij^2}$. 
For a symmetric positive definite matrix $\mB$, we refer to $\norm{\vy}_\mB := \sqrt{\langle \vy, \mB \vy \rangle}$ as the $\mB$-norm of a vector $\vy$.
When a vector norm is used without a subscript, it is the standard Euclidean norm $\norm{\vx} = \sqrt{\sum_{j=1}^n x_j^2}$.

\subsection{Contributions}
\label{contributions}

Our main contributions are threefold: firstly, we show that despite REK involving both sketches and projections in its updates, it is likely not a member of the SAP family of methods when considering applying SAP to a suitable embedded problem. We develop intuition to suggest that it is unlikely that the SAP framework allows for paired-update methods, such as REK, where one update behaves independently of the other. This indicates that the SAP framework is not as all-inclusive as believed, and suggests there is a more general framework to be found that encompasses these paired-update methods.

Secondly, we introduce a new family of methods, SAP-REK$(\epsilon)$, contained in the SAP framework. We show that upon taking $\epsilon \to 0$ one recovers the REK update, thus showing that REK is, in a sense, a limit point of SAP. However, one does not recover a meaningful convergence result for REK through this limit, again suggesting the theory has gaps to be filled.

Finally, we give an extended analysis of SAP-REK$(\epsilon)$. Applying theory from SAP, we derive a bound on the convergence rate of SAP-REK($\epsilon$) for $\epsilon>0$ analogous to the corresponding convergence result for REK. Additionally, we explore the empirical convergence of SAP-REK($\epsilon$) in the limit as $\epsilon\to 0$ and observe that REK typically outperforms SAP-REK($\epsilon$) after enough iterations have passed.

\section{Background}
\label{background}

In this section we give detailed summaries of REK and SAP. 
In particular, we provide explanations of the methods, pseudocode, and the strongest known convergence results.

\subsection{Randomized Extended Kaczmarz}
\label{rekbackground}
REK (\cref{REKcode}) \citep{rek} finds the least squares solution to an inconsistent system $\mA \vx = \vb$ by employing a paired update scheme where two sequences of iterates are produced. 
Firstly, the projection of $\vb$ onto $\rowspace(\mA)^\perp$, $\vb_{\rowspace(\mA)^\perp}$, is approximated by a sequence of iterates $\zk$. 
The iterates $\zk$ are generated by applying RK to the consistent system $\mA^\trans \vz = \mathbf{0}$ with the initial guess $\zinit = \vb$.
At each iteration, a column $\Aj$ is randomly sampled from $\mA$ with probability $\normsq{\Aj}/\normsq{\mA}_F$ and the RK update is applied
\begin{equation*}
    \zkpo = \zk - \frac{\Aj^\trans \zk}{\norm{\Aj}^2}\Aj,
\end{equation*}
where $j$ is the index of the chosen column.

\begin{algorithm}
	\caption{Randomized Extended Kaczmarz}\label{REKcode}
	\begin{algorithmic}[1]
		\Procedure{REK}{$\mA,\vb$, N, $\xinit$}
		\State Initialize $\zinit = \vb$
		\For{$k = 1, 2, \ldots, N-1$}
		    \State Sample $i$ from $\{1, 2, \ldots, m\}$ with probability $\normsq{\Ai}/\normsq{\mA}_F$
		    \State Sample $j$ from $\{1, 2, \ldots, n\}$ with probability $\normsq{\Aj}/\normsq{\mA}_F$
    		\State Update $\zkpo = \zk - \frac{\Aj^\trans \zk}{\norm{\Aj}^2}\Aj$
    		\State Update $\xkpo = \xk - \frac{\Ai\xk - \bi + \zkpoi}{\norm{\Ai}^2} \Ai^\trans$
		\EndFor
		\State \Return $\vx^N$
		\EndProcedure
	\end{algorithmic}
\end{algorithm}

As $\zk \to \vb_{\rowspace(\mA)^\perp}$, the inconsistent system $\mA \vx = \vb - \zk$ approaches the consistent system $\mA \vx = \vb - \vb_{\rowspace(\mA)^\perp}$.
Moreover, the least norm solution to $\mA \vx = \vb - \vb_{\rowspace(\mA)^\perp}$ is exactly the least squares solution $\xopt$.
To compute this least norm solution, a second sequence of iterates $\xk$ is generated by applying RK to the intermediate systems $\mA \vx = \vb -\zk$ with the initial guess $\xinit = \mathbf{0}$ (sometimes  referred to as \textit{randomized orthogonal projection} \citep{rek}). 
At each iteration, a row $\Ai$ is randomly sampled from $\mA$ with probability $\normsq{\Ai}/\normsq{\mA}_F$ and the RK update is applied:
\begin{equation*}
    \xkpo = \xk - \frac{\Ai\xk - \bi + \zkpoi}{\norm{\Ai}^2} \Ai^\trans.
\end{equation*}

We note that one may choose any of the $\vz$ iterates to use, since they are computed independently of the $\vx$ iterates, with the most common choices being $\zk$ and $\zkpo$ \citep{rek, du_tighter_2019}. We will consistently use $\zkpo$ in the computation of $\xkpo$ in this paper, and refer the reader to \citep{dumitrescu_relation_2015} for further discussion.





When proposing the method, \citet{rek} provided a convergence result for REK guaranteeing linear convergence in mean squared error even for inconsistent systems.
Later, \citet{du_tighter_2019} showed the following improved convergence result for REK, which is the tightest currently known bound.
\begin{theorem}[\citet{du_tighter_2019}]\label{thm:rek-convergence}
    Let $\mA \in \mathbb{R}^{m \times n}, \vb \in \mathbb{R}^m$. Let $\xk$ be the $k\thup$ iterate of REK (\cref{REKcode}), $\xopt$ be the least-squares solution of the system $\mA \vx = \vb$. Let us also assume $\xinit \in \text{range}(\mA^\trans)$ and $\zinit \in \vb+ \text{range}(\mA)$. Finally, define $\rho_k = \left(1 -\frac{\sigma^2_{\min}(\mA)}{\norm{\mA}_F^2}\right)^k$. Then we have: 
    \begin{equation*}
        \Eb{\norm{\xk - \xopt}^2} \leq  \rho_k \norm{\xinit - \xopt}^2 + \frac{\rho_k(1-\rho_k)}{\sigma_{\min}(\mA)}\norm{\zinit - (\mI - \mA \mA^\dagger)\vb}^2.
    \end{equation*}
\end{theorem}



\subsection{Sketch-and-Project}
\label{sapbackground}



The Sketch-and-Project family of methods (SAP) (\cref{SAPcode}), introduced by Gower and Richt\'arik~\citep{gower_randomized_2015}, find the solution to a consistent system $\mM \vy = \vc$, with $\mM \in \R^{m\times n}$ and $\vc \in \R^{m}$, by producing a sequence of iterates $\yinit$, $\yone$, $\ytwo$, $\ldots$ that converge to the solution $\yopt$.
SAP methods are parametrized by a symmetric positive definite matrix $\mB \in \R^{m \times n}$ and a distribution $\cD$ over sketching matrices $\mS$ with $m$ rows.
At each iteration, SAP samples a sketching matrix $\mS \sim \cD$ and projects the iterate $\yk$ onto the solution space of the sketched system $\mS^\trans \mM \vy = \mS^\trans \vc$ with respect to the $\mB$-norm $\norm{\vy}_\mB := \sqrt{\langle \vy, \mB \vy \rangle}$.
That is,
\begin{equation}
\label{eqn:sap-optimization-update}
    \ykpo = \argmin_{\vy} \norm{\vy - \yk}_\mB \qtext{subject to} \mS^\trans \mM \vy = \mS^\trans \vc.
\end{equation}
This update can be formulated explicitly as 
\begin{equation}\label{eqn:sap-update}
    \ykpo = \yk - \mB^{-1} \mM^\trans \mS (\mS^\trans \mM \mB^{-1} \mM^ \trans \mS)^\dagger \mS^\trans \parens{\mM \yk - \vc}.
\end{equation}


\begin{algorithm}
	\caption{Sketch-and-Project}\label{SAPcode}
	\begin{algorithmic}[1]
		\Procedure{SAP}{$\mM,\vc, \mathcal{D}, \mB, \yinit, N$}
		\For{k = 1, \ldots, N}
		    \State Sample $\mS \sim \mathcal{D}$
    		\State Update $\ykpo = \yk - \mB^{-1} \mM^\trans \mS (\mS^\trans \mM \mB^{-1} \mM ^\trans \mS)^\dagger \mS^\trans (\mM \yk - \vc)$
		\EndFor
		\State \Return $\vy_N$
		\EndProcedure
	\end{algorithmic}
\end{algorithm}




For particular choices of the matrices $\mB$ and $\mS$, SAP recovers RK and other well-known randomized iterative methods for solving linear systems \citep{gower_randomized_2015}. 
For example, when $\mB = \mI$ and $\mS = \mI_{:i}$ (with row $i$ being sampled with probability proportional to its Euclidean norm) SAP recovers RK. 
Other choices for $\mB$ and $\mS$ allow SAP to recover block RK \citep{needell2014paved}, randomized coordinate descent \citep{leventhal_randomized_2010}, random pursuit \citep{nesterov2011random}, stochasic Newton \citep{qu2016sdna} and randomized block coordinate descent \citep{richtarik2014iteration}.
In addition to recovering the update formulae for such special cases, the general convergence theorem below recovers the best known convergence results for each of these \citep{gower_randomized_2015}. The SAP framework thus provides a powerful unification of modern theory on iterative projection-based methods for linear systems.


\begin{theorem}[\citet{gower_randomized_2015}]
    \label{thm:sapconvergence}
    Let $\yk$ denote the $k$th iterate of SAP (Algorithm 2), and let $\yopt$ be any solution to $\mM \vy = \vc$. If $\Eb{\mZ}$ is positive definite, then
    \begin{equation}\label{eqn:convrate}
        \Eb{\normsq{\yk - \yopt}_{\mB}} \leq \parens{1-\lambda_{\min}(\mB^{-1/2} \Eb{\mZ}\mB^{-1/2})}^k \normsq{\yinit - \yopt}_\mB,
    \end{equation}
    where $\mZ = \mM^\trans \mS (\mS^\trans \mM \mB^{-1} \mM ^\trans \mS)^\dagger \mS^\trans \mM$.
\end{theorem}

In order to apply \cref{thm:sapconvergence} to derive a corollary for a special case of SAP, one need only compute the smallest eigenvalue of $\mB^\inv \Eb{\mZ}$. Taking RK as an example, \cref{rkconvergence} results from \cref{thm:sapconvergence} by computing $\lambda_{\min}(\mB^{-1} \Eb{\mZ}) = \sigma_{\min}^2(\mA) / \normsq{\mA}_F$.

\section{Bridging the gap between SAP and REK}
\label{bridging}



In this section we embed the problem of expressing REK as a SAP method into a suitable consistent system to which SAP may be applied. We provide intuition that it is unlikely that REK may be expressed as a SAP method, however, we show that REK may in fact be recovered as the limit point (in a certain sense) of a family of SAP methods.

\subsection{A suitable embedding}
\label{rekisnotsap}
To begin, we embed the problem as a consistent system of saddle point type \citep{benzi2005saddle},
\begin{equation}\label{eqn:consistent-system}
    \underbrace{\begin{bmatrix}
        \mA^\trans  &  0 \\
             \mI  & \mA
    \end{bmatrix}}_{\eqdef \mM}
    \underbrace{\zx}_{\eqdef \vy}
    = 
    \underbrace{\begin{bmatrix}
        0 \\ \vb
    \end{bmatrix}}_{\eqdef \vc},
\end{equation}
which has solution $\vecstack{\vb_{\rowspace(\mA)^\perp}}{\xopt}$, where $\vb_{\rowspace(\mA)^\perp}$ is the component of $\vb$ in $\rowspace(\mA)^\perp$. We can write the paired REK updates
\begin{equation*}
    \begin{split}
        \zkpo &= \zk - \frac{\Aj^\trans \zk}{\normsq{\Aj}}\Aj \\
        \xkpo &= \xk - \frac{\Ai \xk - \bi + \zkpoi}{\normsq{\Ai}}\Ai^\trans
    \end{split}
\end{equation*}
as a single update in the spirit of SAP as:
\begin{equation}
\label{rekmatrixupdate}
    \zxkpo
    = 
    \zxk
    - 
    \begin{bmatrix}
        \frac{\Aj}{\normsq{\Aj}} & 0 \\
        0 & \frac{\Ai^\trans}{\normsq{\Ai}}
    \end{bmatrix}
    \parens{
    \begin{bmatrix}
        \Aj^\trans & 0 \\
        \mI_\rowi\parens{\mI - \frac{\Aj\Aj^\trans }{\normsq{\Aj}}} & \Ai
    \end{bmatrix}
        \zxk
        - \rhsi
    }.
\end{equation}

Note that in particular, the $\vz$ updates are independent of the $\vx$ updates. Consider now applying a SAP method with sketching distribution $\mathcal{D}$ and projection matrix $\mB$ to the system given in \cref{eqn:consistent-system}. We believe that since $\mB$ is symmetric, if such a method updates $\vz$ independently of $\vx$, then it must in turn update $\vx$ independently of $\vz$ - hence removing the possibility that REK may be expressed as a SAP method on this system, as the $\vx$ updates for REK certainly depend on $\vz$. This may be shown immediately in the case that $\mB$ is block diagonal (a common scenario for SAP methods, e.g. for RK and randomized block Kaczmarz, see \citep{gower_randomized_2015}), and we leave the more general case for future work.

\begin{prop}
Consider applying a SAP method with sketching distribution $\mathcal{D}$ and projection matrix
\[
\mB = \begin{bmatrix} \mB_1 & 0 \\ 0 & \mB_2 \end{bmatrix}
\]
to the system \ref{eqn:consistent-system}. Then if $\vz$ updates independently of $\vx$, $\vx$ must in turn update independently of $\vz$.
\end{prop}

\begin{proof}
Fix $k$, and at iteration $k+1$ suppose that $\mS$ is the sketching matrix sampled from $\mathcal{D}$. Recall that the SAP update takes the form
\[
\vecstack{\zkpo}{\xkpo} = \vecstack{\zk}{\xk}- \mB^{-1}\mM^\trans \mS(\mS^\trans \mM \mB^{-1}\mM^\trans \mS)^\dagger \mS^\trans \left(\mM \vecstack{\zk}{\xk} - \vecstack{0}{\vb}\right).
\]
This may then be written as 
\[
\vecstack{\zkpo}{\xkpo} = \vecstack{\zk}{\xk} - \mB^{-1}\mZ \vecstack{\zk}{\xk} + \vc_1,
\]
where $\mZ = \mM^\trans \mS (\mS^\trans\mM\mB^{-1}\mM^\trans\mS)^\dagger \mS^\trans\mM$ and $\vc_1$ is a vector independent of $\zk$ and $\xk$. 
Using that $\mZ$ is symmetric, we write
\[
\mZ = \begin{bmatrix} \mZ_1 & \mZ_2 \\ \mZ_2^\trans & \mZ_3 \end{bmatrix}.
\]
We then have that
\[
\mB^{-1}\mZ = \begin{bmatrix} \mB_1^{-1}\mZ_1 & \mB_1^{-1}\mZ_2 \\ \mB_2^{-1} \mZ_2^\trans & \mB_2^{-1} \mZ_3 \end{bmatrix}.
\]
Now, in order for $\vz$ to update independently of $\vx$, we must have $\mB_1^{-1}\mZ_2 = 0$, thus $\mZ_2 = 0$. Hence we have
\[
\mB^{-1}\mZ = \begin{bmatrix} \mB_1^{-1}\mZ_1 & 0 \\ 0 & \mB_2^{-1}\mZ_3 \end{bmatrix},
\]
and thus $\vx$ updates independently of $\vz$.
\end{proof}

\begin{remark}
Though not the focus of this work, we remark that this proposition excludes a more general class of methods from the SAP framework on this particular embedded problem. In particular, it is not restricted to methods utilizing only a single row and column, hence also excludes a variety of extended block methods.
\end{remark}

\subsection{SAP-REK$(\epsilon)$ and its limit}
\label{SAP-REKlimit}
We discussed in the previous subsection that it is unlikely that one may express REK as a SAP method for a suitably embedded problem. In this section, we define a new family of SAP methods SAP-REK$(\epsilon)$ depending on a parameter $\epsilon >0 $, and show that in the limit $\epsilon \to 0$ the REK update is recovered. We make use of the same embedding into a consistent system as before \ref{eqn:consistent-system}, and recall that the REK update takes the following form:
\begin{equation}
\label{rekmatrixupdate}
    \zxkpo
    = 
    \zxk
    - 
    \begin{bmatrix}
        \frac{\Aj}{\normsq{\Aj}} & 0 \\
        0 & \frac{\Ai^\trans}{\normsq{\Ai}}
    \end{bmatrix}
    \parens{
    \begin{bmatrix}
        \Aj^\trans & 0 \\
        \mI_\rowi\parens{\mI - \frac{\Aj\Aj^\trans }{\normsq{\Aj}}} & \Ai
    \end{bmatrix}
        \zxk
        - \rhsi
    }.
\end{equation}

We define a related family of SAP methods SAP-REK$(\epsilon)$ depending on a parameter $\epsilon > 0$, with sketching and projection matrices
\begin{equation*}
    \mS = \begin{bmatrix} \mI_\colj & 0 \\ 0 & \mI_{:i} \end{bmatrix} \text{   and   } \mB_\epsilon = \begin{bmatrix} \mI & 0 \\ 0 & \epsilon \mI \end{bmatrix}
\end{equation*}
respectively, where $\mS$ is sampled via sampling $i$ and $j$ independently from $[m]$ and $[n]$ respectively, with probabilities
\begin{equation*}
    \mathbb{P}(\text{Select row $i$}) = \frac{1 + \frac{1}{\epsilon}\norm{\Ai}^2}{m + \frac{1}{\epsilon}\norm{\mA}_{\text{F}}^2} \quad \text{  and  } \quad \mathbb{P}(\text{Select column $j$}) = \frac{\norm{\Aj}^2}{\norm{\mA}_{\text{F}}^2}.
\end{equation*}

\begin{algorithm}
	\caption{SAP-REK$(\epsilon)$}\label{SAPREKcode}
	\begin{algorithmic}[1]
		\Procedure{SAP-REK$(\epsilon)$}{$\mA,\vb, \xinit,\zinit, N$}
		\State Set $\mM = \begin{bmatrix}
		    \mA^\trans & 0 \\ \mI & \mA 
		\end{bmatrix}$
		\State Set $\mB_\epsilon = \begin{bmatrix}
    	        \mI & 0 \\ 0 & \epsilon \mI
    	    \end{bmatrix}$
		\For{k = 1, \ldots, N}
		    \State Sample $i \in \{1, \cdots, m\}$ w.p. $\frac{1 + \frac{1}{\epsilon}\norm{\Ai}^2}{m + \frac{1}{\epsilon}\norm{\mA}_{\text{F}}^2}$
    		\State Sample $j \in \{1, \cdots, n\}$ w.p. $\frac{\norm{\Aj}^2}{\norm{\mA}_{\text{F}}^2}$
    		\State Set $\mS = \begin{bmatrix}
    		    \mI_\colj & 0 \\ 0 & \mI_{:i}
    		\end{bmatrix}$
    	    
    		\State Update $\vecstack{\zkpo}{\xkpo} = \vecstack{\zk}{\xk} - \mB_\epsilon^{-1} \mM^\trans \mS(\mS^\trans \mM \mB_\epsilon^{-1} \mM^\trans \mS)^\dagger \mS^\trans \left(\mM \vecstack{\zk}{\xk} - \vecstack{0}{\vb}\right)$
		\EndFor
		\State \Return $\vx^N$
		\EndProcedure
	\end{algorithmic}
\end{algorithm}
Pseudocode for the method is given in \cref{SAPREKcode}. We provide convergence analysis of this family of methods in \cref{SAP-REKanalysis}, whilst here we show that we recover REK in the limit $\epsilon \to 0$.

Indeed, firstly note that in said limit our sampling probabilities become
\begin{equation*}
    \mathbb{P}(\text{Select row $i$}) = \frac{\norm{\Ai}^2}{\norm{\mA}_{\text{F}}^2} \quad \text{  and  } \quad \mathbb{P}(\text{Select column $j$}) = \frac{\norm{\Aj}^2}{\norm{\mA}_{\text{F}}^2},
\end{equation*}
which are exactly the sampling probabilities for REK, as stated in \cref{rekbackground}. Now, computing the SAP update (\ref{eqn:sap-update}) with SAP-REK$(\epsilon)$'s $\mS$ and $\mB_\epsilon$ yields

\begin{align*}
    \zxkpo
    &= \zxk
    - 
    \begin{bmatrix}
        \mA_\colj & \mI_\rowi^\trans \\
        0 & \frac{1}{\epsilon}\mA_\rowi^\trans
    \end{bmatrix}
    \parens{
        \begin{bmatrix}
            \mA_\colj^\trans & 0 \\
            \mI_\rowi & \mA_\rowi
        \end{bmatrix}
        \begin{bmatrix}
            \mA_\colj & \mI_\rowi^\trans \\
            0 & \frac{1}{\epsilon}\mA_\rowi^\trans
        \end{bmatrix}
    }^\dagger
    \parens{
        \begin{bmatrix}
            \mA_\colj^\trans & 0 \\
            \mI_\rowi & \mA_\rowi
        \end{bmatrix}
        \zxk - \rhsi
    } \\
    &= \zxk
    - 
    \begin{bmatrix}
        \mA_\colj & \mI_\rowi^\trans \\
        0 & \frac{1}{\epsilon}\mA_\rowi^\trans
    \end{bmatrix}
    \parens{
        \begin{bmatrix}
            \normsq{\mA_\colj} & \mA_\subij \\
            \mA_\subij & 1 + \frac{1}{\epsilon}\normsq{\mA_\rowi}
        \end{bmatrix}
    }^\dagger
    \parens{
        \begin{bmatrix}
            \mA_\colj^\trans & 0 \\
            \mI_\rowi & \mA_\rowi
        \end{bmatrix}
        \zxk - \rhsi
    } \\
    &= \zxk
    - 
    \begin{bmatrix}
        \mA_\colj & \mI_\rowi^\trans \\
        0 & \frac{1}{\epsilon}\mA_\rowi^\trans
    \end{bmatrix}
    \frac{\begin{bmatrix}
        1 + \frac{1}{\epsilon}\normsq{\mA_\rowi} & -\mA_\subij \\
        -\mA_\subij & \normsq{\mA_\colj}
    \end{bmatrix}}{\normsq{\mA_\colj} \parens{1 + \frac{1}{\epsilon}\normsq{\mA_\rowi}} - \mA_\subij^2}
    \parens{
        \begin{bmatrix}
            \mA_\colj^\trans & 0 \\
            \mI_\rowi & \mA_\rowi
        \end{bmatrix}
        \zxk - \rhsi
    } \\
    &= \zxk
    - 
    \frac{\begin{bmatrix}
        \parens{1 + \frac{1}{\epsilon}\normsq{\mA_\rowi}} \mA_\colj - \mA_\subij \mI_\rowi^\trans & -\mA_\subij \mA_\colj + \normsq{\mA_\colj} \mI_\rowi^\trans \\
        - \frac{1}{\epsilon} \mA_\subij \mA_\rowi^\trans & \frac{1}{\epsilon} \normsq{\mA_\colj} \mA_\rowi^\trans
    \end{bmatrix}}{\normsq{\mA_\colj} \parens{1 + \frac{1}{\epsilon}\normsq{\mA_\rowi}} - \mA_\subij^2}
    \parens{
        \begin{bmatrix}
            \mA_\colj^\trans & 0 \\
            \mI_\rowi & \mA_\rowi
        \end{bmatrix}
        \zxk - \rhsi
    }.
\end{align*}
In the limit of $\epsilon \to 0$, this simplifies to
\begin{align*}
    \zxkpo
    &= \zxk
    - \begin{bmatrix}
        \frac{\mA_\colj}{\normsq{\mA_\colj}} & 0 \\
        - \frac{\mA_\subij \mA_\rowi^\trans}{\normsq{\mA_\rowi}\normsq{\mA_\colj}} & \frac{\mA_\rowi^\trans}{\normsq{\mA_\rowi}}
    \end{bmatrix}
    \parens{
        \begin{bmatrix}
            \mA_\colj^\trans & 0 \\
            \mI_\rowi & \mA_\rowi
        \end{bmatrix}
        \zxk - \rhsi
    }\\
    &= \zxk
    - 
    \begin{bmatrix}
        \frac{\Aj}{\normsq{\Aj}} & 0 \\
        0 & \frac{\Ai^\trans}{\normsq{\Ai}}
    \end{bmatrix}
    \parens{
    \begin{bmatrix}
        \Aj^\trans & 0 \\
        \mI_\rowi\parens{\mI - \frac{\Aj\Aj^\trans }{\normsq{\Aj}}} & \Ai
    \end{bmatrix}
        \zxk
        - \rhsi
    }
\end{align*}
which is exactly the REK update given in \cref{rekmatrixupdate}. Therefore, REK is recovered exactly from the SAP-REK$(\epsilon)$ family when $\epsilon \to 0$, and thus REK is a limit point of the Sketch-and-Project family of methods, in this particular sense.

\begin{remark}
Note that taking $\epsilon \to 0$ truly takes us outside of the SAP regime, as SAP requires the projection matrix $\mB$ to be invertible.
\end{remark}

However, The convergence result from REK is unfortunately not recovered in this limit. Applying \cref{thm:sapconvergence} to SAP-REK$(\epsilon)$, we obtain
\begin{equation*}
    \Eb{\norm{\zk - \vb_{\rowspace(\mA)^\perp}}^2 + \epsilon \norm{\xk - \xopt}^2} \leq \parens{1 - \lambda_{\min}(\mB_\epsilon^{-1/2}\Eb{\mZ}\mB_\epsilon^{-1/2})}^k\left(\norm{\zinit - \vb_{\rowspace(\mA)^\perp}}^2 + \epsilon \norm{\xk - \xopt}^2\right).
\end{equation*} 
Upon taking $\epsilon \to 0$, any information about the convergence of the $\vx$ iterates is lost. In the following section, we examine the convergence of SAP-REK$(\epsilon)$ more thoroughly.

\section{Analysis of SAP-REK$(\epsilon)$}
\label{SAP-REKanalysis}

In this section we analyze the SAP-REK$(\epsilon)$ family of methods. We provide theoretical guarantees on convergence, and show how several steps along the way produce the corresponding results for REK in the limit $\epsilon \to 0$. Firstly, we state our convergence result:

\begin{theorem}
Let $\mA \in \mathbb{R}^{m \times n}, \vb \in \mathbb{R}^m$, $\epsilon > 0$. Denote by $\xopt$ the least-squares solution of the system $\mA \vx = \vb$, and denote by $\xk, \zk$ the $k\thup$ iterates of SAP-REK$(\epsilon)$ applied to said system with initial iterates $\zinit = \vb, \xinit \in \rowspace(\mA)$. Assume for convenience that $\mA$ has been normalized such that $\norm{\mA}_F = 1$. Then we have

\begin{equation*}
    \Eb{\norm{\zk - \vb_{\rowspace(\mA)^\perp}}^2 + \epsilon\norm{\xk - \xopt}^2} \leq \left(1- \lambda_{\min}^{+}(\mW'_\epsilon)\right)^k\left(\norm{\zinit - \vb_{\rowspace(\mA)}}^2 + \epsilon \norm{\xinit - \xopt}^2\right),
\end{equation*}
where $\mM$ is as in \cref{SAP-REKlimit}, and $\lambda_{\min}^{+}(\mW'_\epsilon)$ is given by
\begin{equation*}
    \min\left(\frac{1}{m+\frac{1}{\epsilon}},  \frac{\left(m-2 + \frac{2}{\epsilon}\right)\sigma^2 + 1}{2\left(m + \frac{1}{\epsilon}\right)} - \frac{1}{m + \frac{1}{\epsilon}}\left(\frac{\left(\left(m-2+\frac{2}{\epsilon}\right)\sigma^2 + 1\right)^2}{4} - \frac{1}{\epsilon}\sigma^4\left(\left(m + \frac{1}{\epsilon}\right) - \sigma^2\right)\right)^{1/2}\right),
\end{equation*}
where $\sigma = \sigma_{\min}(\mA)$.
\end{theorem}We follow the convergence result for SAP methods provided in \citep{gower_randomized_2015}, namely that given an SAP method with update matrix $\mZ = \mM^\trans \mS \parens{\mS^\trans \mM \mB^\inv \mM^\trans \mS}^\pinv \mS^\trans \mM$, the convergence rate depends on the smallest positive eigenvalue of $\mW := \Eb{\mB^{-1/2}\mZ\mB^{-1/2}}$:
\begin{equation*}
    \Eb{\norm{\vecstack{\zk - \zopt}{\xk - \xopt}}^2_{\mB}} \leq \left(1-\lambda_{\min}^{+}(\mW)\right)^k\norm{\vecstack{\zinit - \zopt}{\xinit - \xopt}}^2.
\end{equation*}

We provide computations of $\mZ$, $\mW$, and the convergence rate for SAP-REK$(\epsilon)$.

Firstly, let $\mZ_{\epsilon}$ be the update matrix $\mM^\trans \mS \parens{\mS^\trans \mM \mB_\epsilon^\inv \mM^\trans \mS}^\pinv \mS^\trans \mM$ when taking $\mB_\epsilon$ and $\mS$ as in \cref{SAP-REKlimit}, and compute this in block matrix form as:

\begin{equation*}
\resizebox{\textwidth}{!}{$ \frac{1}{\norm{\Aj}^2(1 + \frac{1}{\epsilon}\norm{\Ai}^2) - \mA_\subij^2}\begin{bmatrix}(1 + \frac{1}{\epsilon}\norm{\Ai}^2)\Aj \Aj^\trans - \mA_\subij \Aj \mI_\rowi - \mA_\subij \mI_\rowi^\trans \Aj^\trans + \norm{\Aj}^2\mI_\rowi^\trans \mI_\rowi & -\mA_\subij \Aj \Ai + \norm{\Aj}^2\mI_\rowi^\trans \Ai \\ -\mA_\subij \Ai^\trans \Aj^\trans + \norm{\Aj}^2\Ai^\trans \mI_\rowi & \norm{\Aj}^2\mA_{\rowi}^\trans \Ai \end{bmatrix}$}
\end{equation*}
Multiplying on either side by $\mB_\epsilon^{-1/2}$ yields the matrix whose expectation is $\mW_\epsilon$, the matrix whose smallest positive eigenvalue will determine the convergence rate:
\begin{equation*}
\resizebox{\textwidth}{!}{
    $\frac{1}{\norm{\Aj}^2(1 + \frac{1}{\epsilon}\norm{\Ai}^2) - \mA_\subij^2}\begin{bmatrix}(1 + \frac{1}{\epsilon}\norm{\Ai}^2)\Aj \Aj^\trans - \mA_\subij \Aj \mI_\rowi - \mA_\subij \mI_\rowi^\trans \Aj^\trans + \norm{\Aj}^2\mI_\rowi^\trans \mI_\rowi & \frac{1}{\sqrt{\epsilon}}\left(-\mA_\subij \Aj \Ai + \norm{\Aj}^2\mI_\rowi^\trans \Ai\right) \\ \frac{1}{\sqrt{\epsilon}}\left(-\mA_\subij \Ai^\trans \Aj^\trans + \norm{\Aj}^2\Ai^\trans \mI_\rowi\right) & \frac{1}{\epsilon}\norm{\Aj}^2\mA_{\rowi}^\trans \Ai \end{bmatrix}$}
\end{equation*}

\begin{remark}
Note that taking the limit $\epsilon \to 0$ at this point yields

\begin{align*}
\frac{1}{\norm{\Aj}^2\norm{\Ai}^2}\begin{bmatrix}
    \norm{\Ai}^2\Aj \Aj^\trans & 0 \\ 0 & \norm{\Aj}^2 \Ai^\trans \Ai 
\end{bmatrix} 
= \begin{bmatrix}
    \frac{\Aj \Aj^\trans}{\norm{\Aj}^2} & 0 \\
    0 & \frac{\Ai^\trans \Ai}{\norm{\Ai}^2}
\end{bmatrix}
\end{align*}
Taking the largest eigenvalue of the expectation of this (when rows and columns are sampled with probabilities proportional to their norms, as in REK) gives
\begin{align*}
    \eigmin &= \eigmin\left(\E\begin{bmatrix}
    \frac{\Aj \Aj^\trans}{\norm{\Aj}^2} & 0 \\
    0 & \frac{\Ai^\trans \Ai}{\norm{\Ai}^2}
\end{bmatrix}\right) \\
    &= \eigmin\left(\frac{1}{\norm{\mA}_{\text{F}}^2}\begin{bmatrix} \mA \mA^\trans & 0 \\ 0 & \mA^\trans \mA\end{bmatrix}\right) \\
    &= \frac{\sigma_{\min}^2(\mA)}{\norm{\mA}_{\text{F}}^2}.
\end{align*}

Plugging this into \cref{eqn:convrate} yields
\begin{equation}
    \label{eqn:useless}
    \norm{\zk}^2 \leq \left(1 - \frac{\sigma_{\min}^2(\mA)}{\norm{\mA}_{\text{F}}^2}\right)^k \norm{\zinit}^2.
\end{equation}

This tells us nothing about how the iterates $\xk$ converge, but we have recovered the proven convergence rate for randomized orthogonal projection, see \citep{rek}.

\end{remark}

When computing the expectation of our matrix, the $A_{ij}^2$ term in the denominator of the scalar in front is troublesome. To simplify our calculations, we will drop it - this will only reduce the eigenvalues of our matrix and give a slightly weaker estimate on the convergence rate. The resultant matrix is given by

\begin{equation}\label{eqn:bigwmatrix}
\resizebox{\textwidth}{!}{$\frac{1}{\norm{\Aj}^2(1 + \frac{1}{\epsilon}\norm{\Ai}^2)}\begin{bmatrix}(1 + \frac{1}{\epsilon}\norm{\Ai}^2)\Aj \Aj^\trans - \mA_\subij \Aj \mI_\rowi - \mA_\subij \mI_\rowi^\trans \Aj^\trans + \norm{\Aj}^2\mI_\rowi^\trans \mI_\rowi & \frac{1}{\sqrt{\epsilon}}(-\mA_\subij \Aj \Ai + \norm{\Aj}^2\mI_\rowi^\trans \Ai) \\ \frac{1}{\sqrt{\epsilon}}(-\mA_\subij \Ai^\trans \Aj^\trans + \norm{\Aj}^2\Ai^\trans \mI_\rowi) & \frac{1}{\epsilon}\norm{\Aj}^2\mA_{\rowi}^\trans \Ai \end{bmatrix},$}
\end{equation}
and we refer to its expectation as $\mW'_\epsilon$.

To compute the expectation of \eqref{eqn:bigwmatrix}, we recall the row and column sampling distribution given in \cref{SAP-REKlimit}, namely 
\[
\mathbb{P}(\text{Select column $j$}) = \frac{\norm{\Aj}^2}{\norm{\mA}_{\text{F}}^2}
\]
and
\[
\mathbb{P}(\text{Select row $i$}) = \frac{1 + \frac{1}{\epsilon}\norm{\Ai}^2}{m + \frac{1}{\epsilon}\norm{\mA}_{\text{F}}^2}.
\]

We also define the following useful matrices that will come up repeatedly in our expectation computation:

\begin{align*}
\prows &= \frac{1}{m+\frac{1}{\epsilon}\norm{\mA}_{\text{F}}^2}\operatorname{diag}\left(1 + \frac{1}{\epsilon}\norm{\Ai}^2\right) \\
\pcols &= \frac{1}{\norm{\mA}_{\text{F}}^2}\operatorname{diag}\left(\norm{\Aj}^2\right) \\
\drows &= \operatorname{diag}\left(1+ \frac{1}{\epsilon}\norm{\Ai}^2\right)\\
\dcols &= \operatorname{diag}\left(\norm{\Aj}^2\right).
\end{align*}
We now compute the expectation of \eqref{eqn:bigwmatrix} block-by-block. For the upper left block, we compute the expectation as
\begin{align*}
\mathbb{E}\left[\text{Top Left Block}\right] &= \mA \pcols \dcols^{-1} \mA^\trans - \mA \pcols \dcols^{-1}\mA^\trans \prows \drows^{-1} - \drows^{-1}\prows \mA \dcols^{-1}\pcols \mA^\trans + \prows \drows^{-1} \\
&= \frac{m-2+\frac{1}{\epsilon}\norm{\mA}_\text{F}^2}{\norm{\mA}_\text{F}^2(m + \frac{1}{\epsilon}\norm{\mA}_\text{F}^2)}\mA\mA^{\trans} + \frac{1}{m+\frac{1}{\epsilon}\norm{\mA}_\text{F}^2}\mI.
\end{align*}
Next, we look at the top right entry. We obtain the expectation as being
\begin{align*}
\mathbb{E}\left[\text{Top Right Block}\right] &= \frac{1}{\sqrt{\epsilon}}(-\mA \pcols \dcols^{-1}\mA^\trans \prows \drows^{-1}\mA + \prows \drows^{-1}\mA)\\
&= \frac{1}{\sqrt{\epsilon}}\left(-\frac{1}{\norm{\mA}_\text{F}^2(m+\frac{1}{\epsilon}\norm{\mA}_\text{F}^2)}\mA\mA^\trans\mA + \frac{1}{m+\frac{1}{\epsilon}\norm{\mA}_\text{F}^2}\mA \right) \\
&= \frac{1}{m+\frac{1}{\epsilon}\norm{\mA}_\text{F}^2}\frac{1}{\sqrt{\epsilon}}\mA\left(\mI - \frac{1}{\norm{\mA}_\text{F}^2}\mA^\trans \mA\right).
\end{align*}
The bottom left block is the transpose of the top right block, so its expectation is

\[
\mathbb{E}\left[\text{Bottom Left Block}\right] = \frac{1}{m+\frac{1}{\epsilon}\norm{\mA}_\text{F}^2}\frac{1}{\sqrt{\epsilon}}\left(\mI - \frac{1}{\norm{\mA}_\text{F}^2}\mA^\trans \mA\right)\mA^\trans.
\]
Finally, the bottom right entry has expectation
\begin{align*}
\mathbb{E}\left[\text{Bottom Right Block}\right] &= \frac{1}{\epsilon}\mA^{\trans}\prows \drows^{-1}\mA \\
&= \frac{1}{m+\frac{1}{\epsilon}\norm{\mA}_\text{F}^2}\frac{1}{\epsilon}\mA^\trans \mA.
\end{align*}
Together this yields
\begin{equation*}
    \mW'_\epsilon = \frac{1}{m+\frac{1}{\epsilon}\norm{\mA}_\text{F}^2}\begin{bmatrix} \frac{(m -2 + \frac{1}{\epsilon}\norm{\mA}_\text{F}^2)}{\norm{\mA}_\text{F}^2}\mA\mA^\trans + \mI & \frac{1}{\sqrt{\epsilon}}\mA\left(\mI - \frac{1}{\norm{\mA}_\text{F}^2}\mA^\trans \mA\right) \\ \frac{1}{\sqrt{\epsilon}}\left(\mI - \frac{1}{\norm{\mA}_\text{F}^2}\mA^\trans \mA\right)\mA^\trans & \frac{1}{\epsilon}\mA^\trans \mA \end{bmatrix}.
\end{equation*}

To determine the convergence rate of SAP-REK$(\epsilon)$, we must determine the smallest positive eigenvalue of this. For ease of reading, we take $\mA$ to be normalized so that $\norm{\mA}_\text{F} = 1$, but the calculations in the general case are similar. Let $\mA = \mU \Sigma \mV^\trans$ be the singular value decomposition of $\mA$, then we have
\begin{equation*}
    \mW'_\epsilon = \frac{1}{m+\frac{1}{\epsilon}}\begin{bmatrix} \mU((m-2+\frac{1}{\epsilon})\Sigma\Sigma^\trans + I)\mU^\trans & \mU(\frac{1}{\sqrt{\epsilon}}(\Sigma - \Sigma\Sigma^\trans\Sigma))\mV^\trans \\ \mV(\frac{1}{\sqrt{\epsilon}}(\Sigma^\trans - \Sigma^\trans\Sigma\Sigma^\trans))\mU^\trans & \mV(\frac{1}{\epsilon}\Sigma^\trans \Sigma)\mV^\trans 
    \end{bmatrix}
\end{equation*}

This can be factorized further as
\begin{equation*}
    \mW'_\epsilon = \begin{bmatrix} \mU & 0 \\ 0 & \mV \end{bmatrix}\frac{1}{m+\frac{1}{\epsilon}}\begin{bmatrix} (m-2+\frac{1}{\epsilon})\Sigma\Sigma^\trans + I & \frac{1}{\sqrt{\epsilon}}(\Sigma - \Sigma\Sigma^\trans\Sigma) \\ \frac{1}{\sqrt{\epsilon}}(\Sigma^\trans - \Sigma^\trans \Sigma\Sigma^\trans) & \frac{1}{\epsilon}\Sigma^\trans \Sigma \end{bmatrix} \begin{bmatrix} \mU^\trans & 0 \\ 0 & \mV^\trans \end{bmatrix}.
\end{equation*}

We are left to find the eigenvalues of the middle matrix above. After row and column swaps, said matrix may be manipulated such that the diagonal consists of $m-n$ entries with value $\frac{1}{m + \frac{1}{\epsilon}}$, and a collection of $2 \times 2$ blocks of the form
\begin{equation*}
    \frac{1}{m+\frac{1}{\epsilon}}\begin{bmatrix} (m -2 + \frac{1}{\epsilon})\sigma_i^2 +1 & \frac{1}{\sqrt{\epsilon}}(\sigma_i - \sigma_i^3) \\ \frac{1}{\sqrt{\epsilon}}(\sigma_i - \sigma_i^3) & \frac{1}{\epsilon}\sigma_i^2 \end{bmatrix}
\end{equation*}

where each $\sigma_i$ is a singular value of $\mA$. The eigenvalues of these blocks are, via standard computation, given by

\begin{equation*}
    \frac{\left(m-2 + \frac{2}{\epsilon}\right)\sigma_i^2 + 1}{2\left(m + \frac{1}{\epsilon}\right)} \pm \frac{1}{m + \frac{1}{\epsilon}}\left(\frac{\left(\left(m-2+\frac{2}{\epsilon}\right)\sigma_i^2 + 1\right)^2}{4} - \frac{1}{\epsilon}\sigma_i^4\left(\left(m + \frac{1}{\epsilon}\right) - \sigma_i^2\right)\right)^{1/2}.
\end{equation*}

We seek the smallest of these, so take the negative option in the $\pm$, and note that for fixed $m$ and $\epsilon$ this is an increasing function of $\sigma^2$. Furthermore, these are always positive. Thus, the smallest eigenvalue of $\mW'_\epsilon$, $\lambda_{\min}^{+}(\mW'_\epsilon)$ is given by

\begin{equation*}
    \min\left(\frac{1}{m+\frac{1}{\epsilon}},  \frac{\left(m-2 + \frac{2}{\epsilon}\right)\sigma^2 + 1}{2\left(m + \frac{1}{\epsilon}\right)} - \frac{1}{m + \frac{1}{\epsilon}}\left(\frac{\left(\left(m-2+\frac{2}{\epsilon}\right)\sigma^2 + 1\right)^2}{4} - \frac{1}{\epsilon}\sigma^4\left(\left(m + \frac{1}{\epsilon}\right) - \sigma^2\right)\right)^{1/2}\right),
\end{equation*}

where $\sigma$ is the smallest singular value of $\mA$. Plugging this into the standard SAP convergence result \cref{thm:sapconvergence} yields

\begin{equation*}
    \Eb{\norm{\vecstack{\zk}{\xk} - \vecstack{\vb_{\rowspace(\mA)^\perp}}{\xopt}}_\mB^2} \leq \left(1- \lambda_{\min}^{+}(\mW'_\epsilon)\right)\norm{\vecstack{\zinit}{\xinit} - \vecstack{\vb_{\rowspace(\mA)^\perp}}{\xopt}}_\mB^2
\end{equation*}

This is equivalent to our claimed convergence result:

\begin{equation}
\label{eq:convergence_result}
\Eb{\norm{\zk - \vb_{\rowspace(\mA)^\perp}}^2 + \epsilon\norm{\xk - \xopt}^2} \leq \left(1- \lambda_{\min}^{+}(\mW'_\epsilon)\right)^k\left(\norm{\zinit - \vb_{\rowspace(\mA)^\perp}}^2 + \epsilon \norm{\xinit - \xopt}^2\right)
\end{equation}

\begin{remark}
We note that in this result, sending $\epsilon$ to 0 directly again recovers \cref{eqn:useless}. One may also be inclined to believe that sending $\epsilon$ to infinity may recover a similar (and thus more useful) statement about the convergence of the $\vx$ iterates - however, our theoretical convergence rate becomes zero in such a limit. Furthermore, taking very large $\epsilon$ will weight that $\vx$ updates so strongly that $\vz$ will update extremely slowly, harming convergence to the least squares solution. We expand more upon this phenomenon in our experiments section.
\end{remark}

\section{Experiments}
\label{experiments}


We perform experiments using random matrices, denoted $\mA$, of size $200 \times 10$ (unless otherwise stated) generated with i.i.d. entries that are either standard Gaussian or uniform, denoted $\mathcal{U}[0,1)$.
Matrices with Gaussian entries are labeled ``Gaussian", while matrices with uniform entries are labeled ``coherent".
For each experiment we run 50 trials.
For each trial, we generate $\vb$ by sampling from $\mathcal{U}[0, 1)$ distribution and apply REK and SAP-REK($\epsilon$) to solve for the least squares solution to the inconsistent system $\mA \vx = \vb$. We furthermore compute $\zopt := \vb_{\rowspace(\mA)^\perp}$.
Depending on the particular experiment, we allow each method to run up to 10,000 iterations.

We perform two types of experiments.
\begin{enumerate}
    \item In the first type of experiment, we plot the mean squared error $\normsq{\xk-\xopt}$, or one of the analogous quantities $\normsq{\zk - \zopt}$ and $\normsq{\zk - \zopt} + \epsilon \normsq{\xk - \xopt}$, as a function of the iteration $k$. We use various fixed values of $\epsilon$ for SAP-REK$(\epsilon)$ ranging from 0.01 to 100 or from 0.0001 to 1 depending on the experiment. Results from such experiments are shown in \cref{fig:few_it_big}, \cref{fig:few_it}, \cref{fig:xs}, \cref{fig:zs}, \cref{fig:xs_zs}.
    \item In the second type of experiment, we plot the mean squared error $\normsq{\xk-\xopt}$ as a function of $\epsilon$. We use various fixed numbers of iterations $k=2500, 5000, 7500, 10000$. We also plot the value of $\lambda_{\min}^+(\mW_\epsilon)$ as a function of $\epsilon$ for Gaussian and coherent matrices. Results from such experiments are shown in \cref{fig:epsilon}.
\end{enumerate}
In each experiment plot, we add a shaded region behind the mean squared error spanning from the $5\thup$ smallest to the $5\thup$ largest observed values of the squared error.
These shaded regions provide some insight about the variance of the squared error.

In \cref{fig:few_it_big} and \cref{fig:few_it}, we study the convergence of SAP-REK$(\epsilon)$ for small $\epsilon$ to see whether its performance approaches that of REK as $\epsilon \to 0$.
In particular, we measure the squared error $\normsq{\xk - \xopt}$ as a function of the iteration $k$ for various values of $\epsilon = 0.0001, 0.001, 0.01, 0.1, 1$.
All of the methods converge slower for the coherent matrix than the Gaussian matrix since the coherent matrix has worse conditioning. 
Also, we see that the convergence of SAP-REK$(\epsilon)$ is similar to that of REK initially, but SAP-REK$(\epsilon)$ plateaus before converging to the solution $\xopt$.
This plateau occurs closer to machine precision as $\epsilon \to 0$ with no noticeable impact on the convergence rate prior to the plateau.
This behavior suggests that REK is preferable to SAP-REK$(\epsilon)$ for small values of $\epsilon$.


\captionsetup[subfigure]{labelformat=empty}

\begin{figure}[H]
\centering
\begin{subfigure}{.48\textwidth} 
  \includegraphics[width=\linewidth]{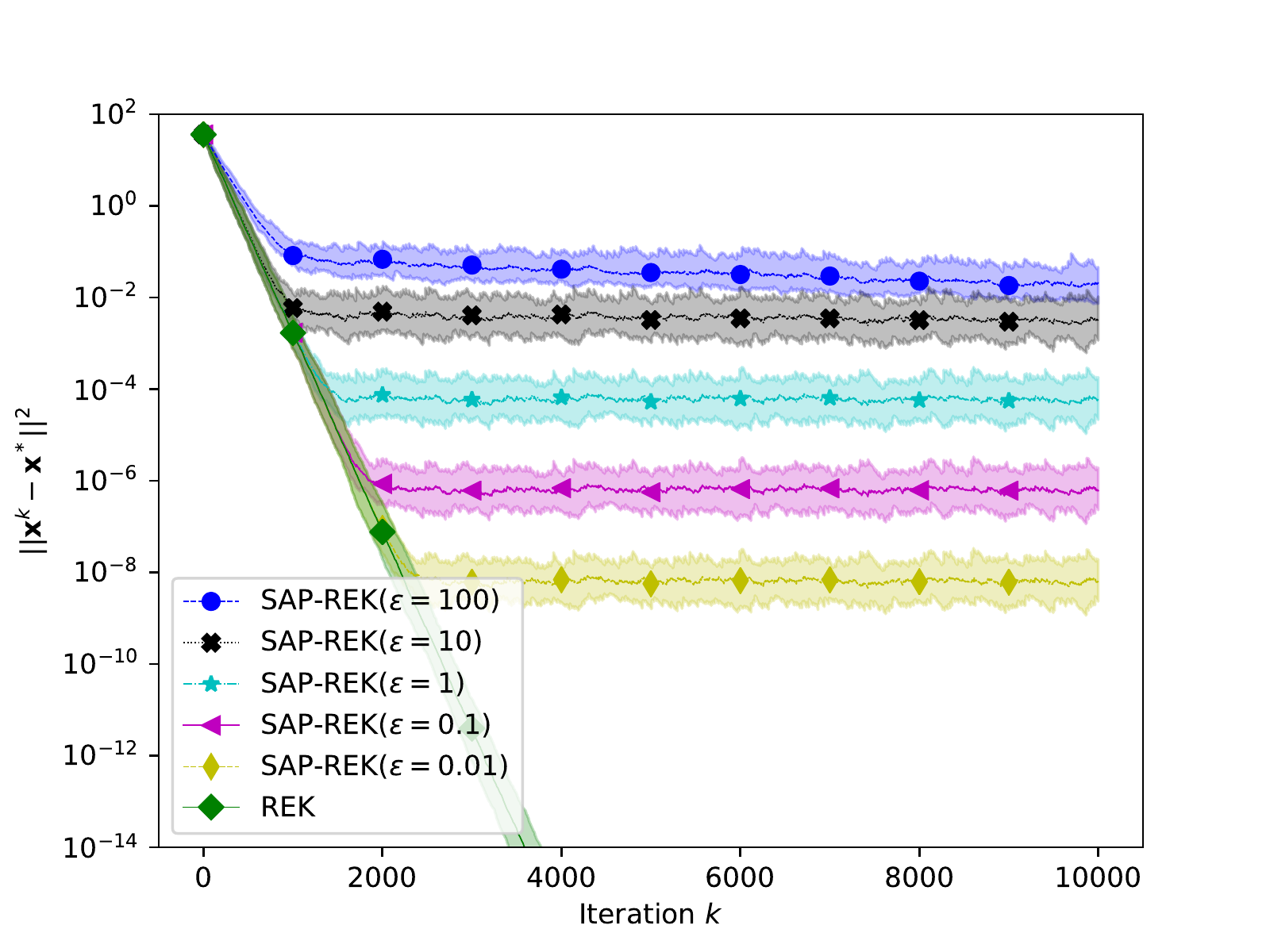}
  \caption{$5000 \times 100$ Gaussian matrix}
\end{subfigure}
\begin{subfigure}{.48\textwidth}  
  \includegraphics[width=\linewidth]{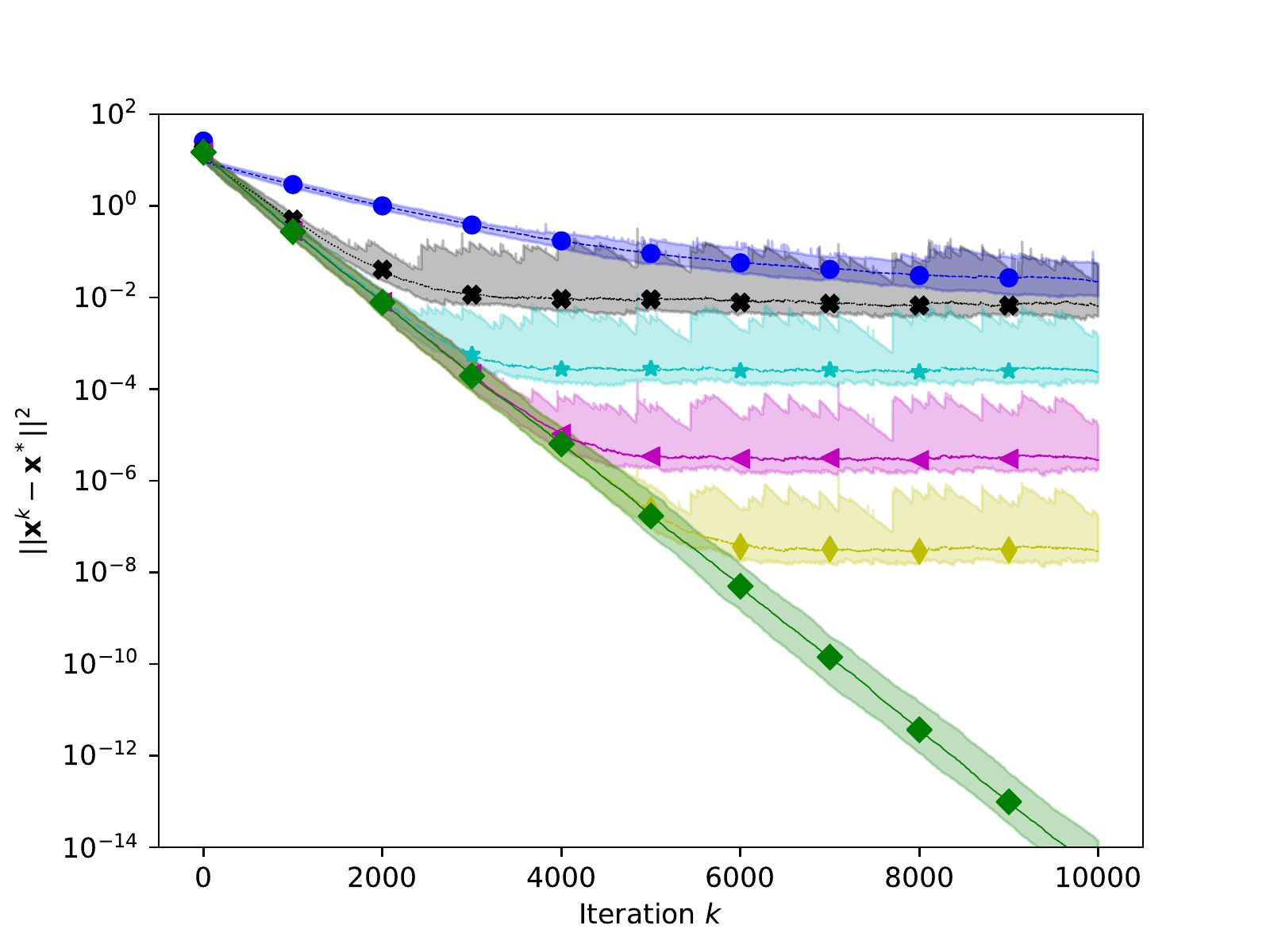}
  \caption{$5000 \times 100$ coherent matrix}
\end{subfigure}
\caption{The convergence rate of the $\vx$ iterates for different values of $\epsilon$ in SAP-REK$(\epsilon)$. We run the algorithms for $1000$ iterations for the Gaussian matrix and coherent matrix. Notice that each SAP-REK$(\epsilon)$ hits a convergence plateau.}
  \label{fig:few_it_big}
\end{figure}

\captionsetup[subfigure]{labelformat=empty}

\begin{figure}[H]
\centering
\begin{subfigure}{.48\textwidth} 
  \includegraphics[width=\linewidth]{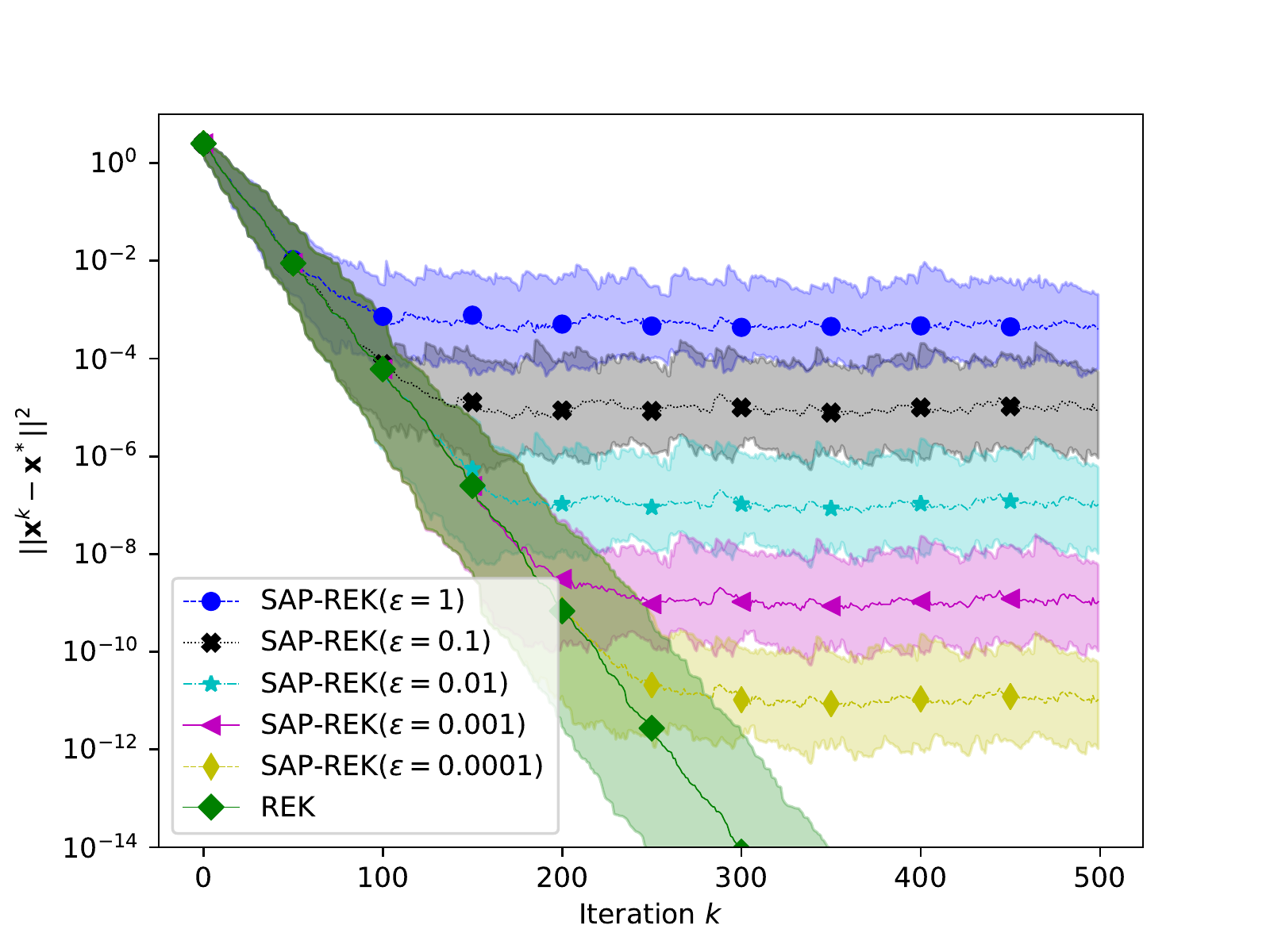}
  \caption{$200 \times 10$ Gaussian matrix}
\end{subfigure}
\begin{subfigure}{.48\textwidth}  
  \includegraphics[width=\linewidth]{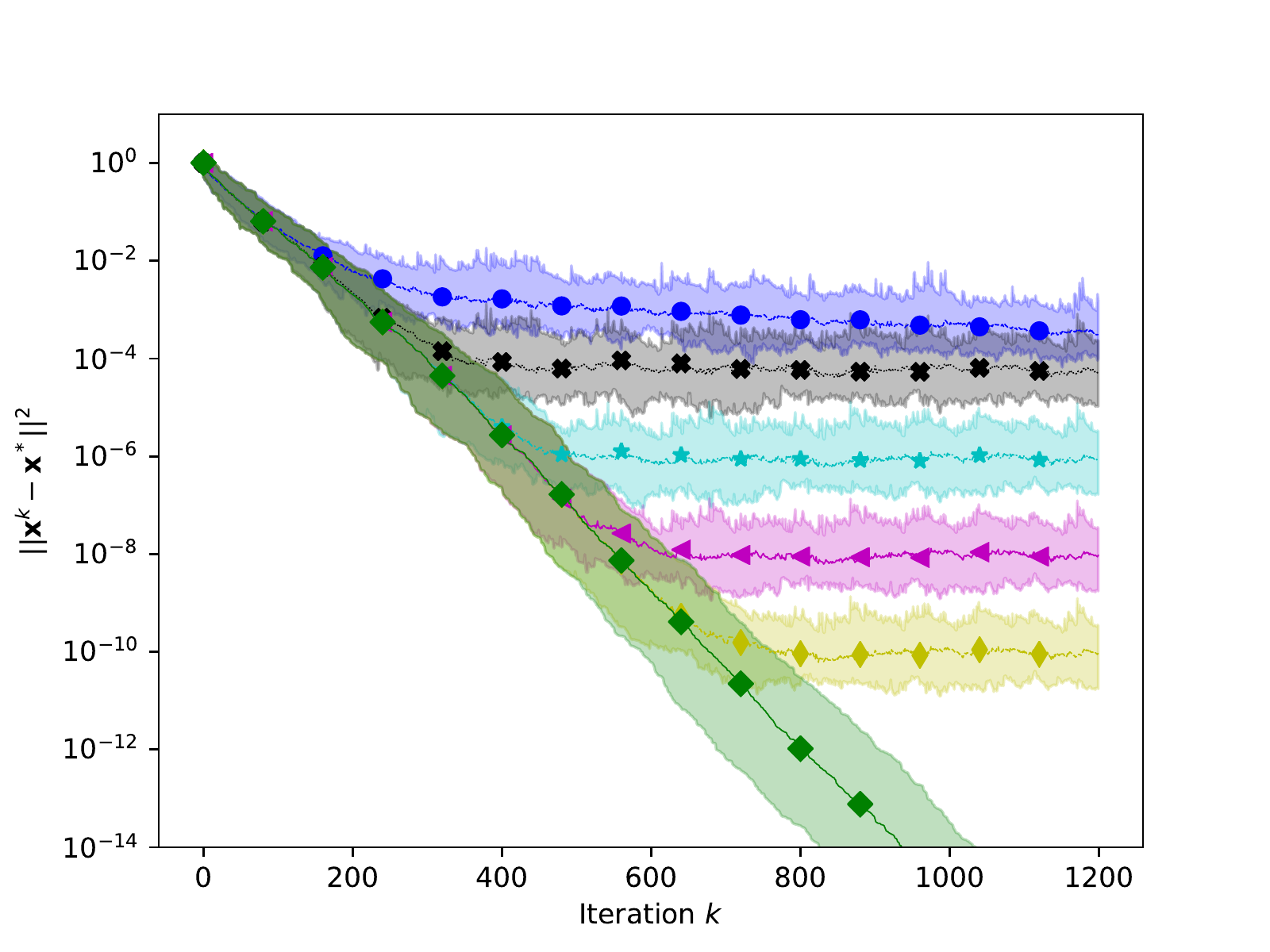}
  \caption{$200 \times 10$ coherent matrix}
\end{subfigure}
\caption{The convergence rate of the $\vx$ iterates for different values of $\epsilon$ in SAP-REK$(\epsilon)$. We run the algorithms for $500$ iterations for the Gaussian matrix and $1200$ iterations for the coherent matrix. Notice that each SAP-REK$(\epsilon)$ hits a convergence plateau similar to that of the $5000 \times 100$ matrices.}
  \label{fig:few_it}
\end{figure}


We proceed with the rest of our experiments using a $200 \times 10$ matrix in order to understand the convergence of SAP-REK$(\epsilon)$ after many more iterations. This investigation is done with the smaller matrix mainly for efficiency since the convergence of SAP-REK$(\epsilon)$ behaves similarly on a small matrix with further iterations (\cref{fig:few_it}) as it does on a large matrix with more iterations (\cref{fig:few_it_big}).

In \cref{fig:epsilon}, we study the convergence of SAP-REK$(\epsilon)$ after many iterations to see whether the algorithm eventually overcomes the plateau observed in \cref{fig:few_it}.
In particular, we measure the squared error $\normsq{\xk-\xopt}$ at a fixed iteration $k$ as a function of the parameter $\epsilon \in [10^{-5}, 10^4]$.
In some experiments we discovered that SAP-REK$(\epsilon)$ for $\epsilon = 1,10,100$ eventually converge, whereas the smaller and larger $\epsilon$ values reach a convergence plateau. This may be due to the $x$ updates being weighted more in SAP-REK$(\epsilon > 1)$. However, when $\epsilon$ is weighted too much (eg. $\epsilon > 100$ for a coherent matrix) we continue to see the $\epsilon$ plateau. 

The most overall convergent SAP-REK$(\epsilon)$ methods out of the $\epsilon$ sampled in \cref{fig:epsilon} were SAP-REK$(\epsilon = 100)$ for the Gaussian matrix and SAP-REK$(\epsilon = 10)$ for coherent matrix. We show this phenomenon is consistent with the convergence result in \cref{eq:convergence_result} by plotting $\lambda^+_{\min}(\mW_\epsilon)$ as a function of $\epsilon$ (respectively labeled Lambda and Epsilon in figure \cref{fig:epsilon}). For the Gaussian matrix, $\lambda^+_{\min}(\mW_\epsilon)$ has a local maximizer near $\epsilon = 100$. For the coherent matrix, $\lambda^+_{\min}(\mW_\epsilon)$ has a local maximizer near $\epsilon = 10$. This suggests that the larger $\lambda^+_{\min}(\mW_\epsilon)$, the faster the convergence of SAP-REK$(\epsilon)$. A larger $\lambda^+_{\min}(\mW_\epsilon)$ corresponds to a smaller $1-\lambda^+_{\min}(\mW_\epsilon)$.  \cref{eq:convergence_result} tells a smaller $1-\lambda^+_{\min}(\mW_\epsilon)$ forces $\Eb{\norm{\zk - \zopt}^2 + \epsilon\norm{\xk - \xopt}^2}$ to be smaller for larger $k$. This also suggests that $\norm{\xk - \xopt}^2$ will be smaller for a larger $\lambda^+_{\min}(\mW_\epsilon)$.

\begin{figure}[H]
    \begin{subfigure}{.48\textwidth}  
        \includegraphics[width=\linewidth]{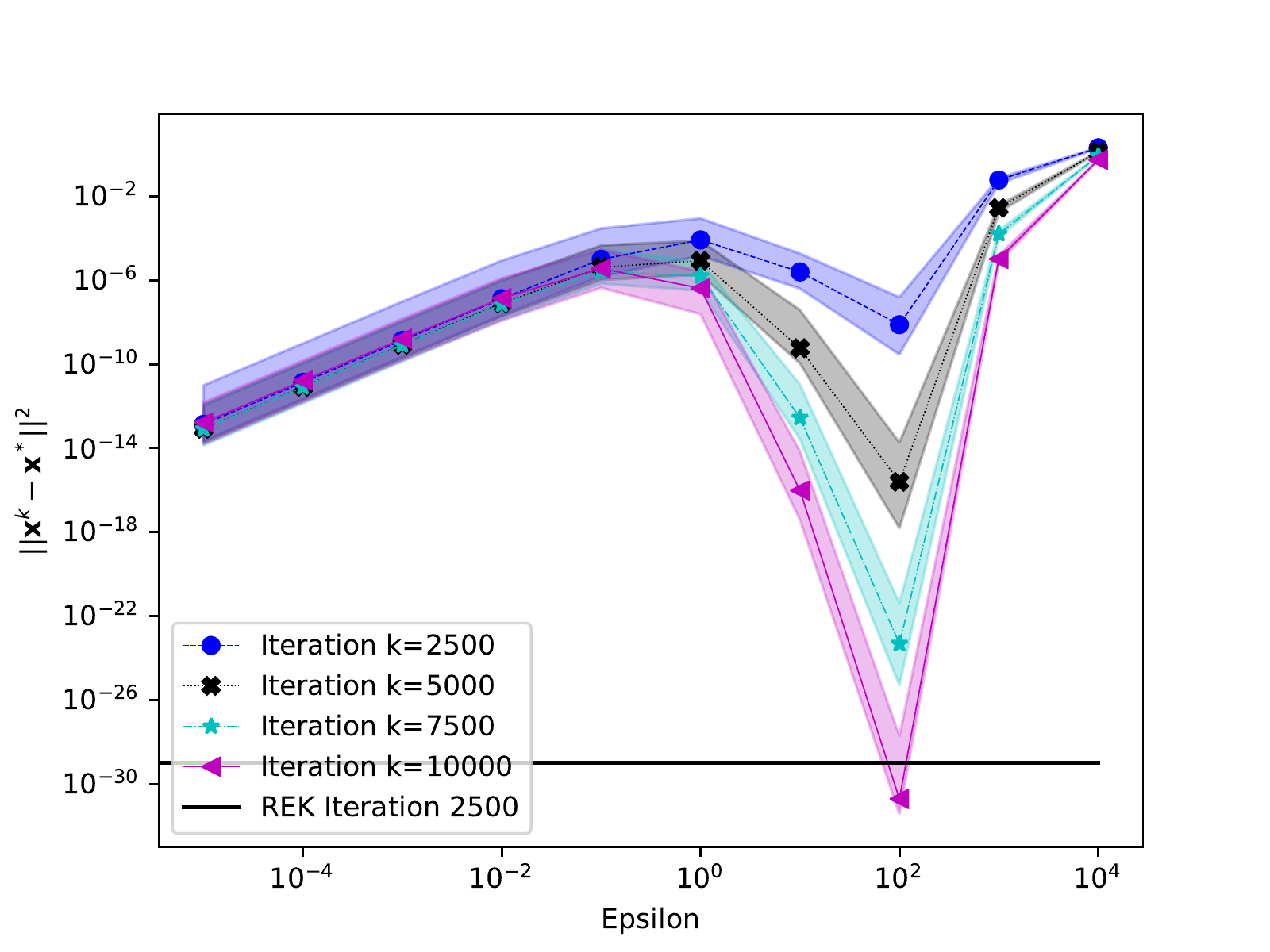}
    \end{subfigure}
    \begin{subfigure}{.48\textwidth} 
        \includegraphics[width=\linewidth]{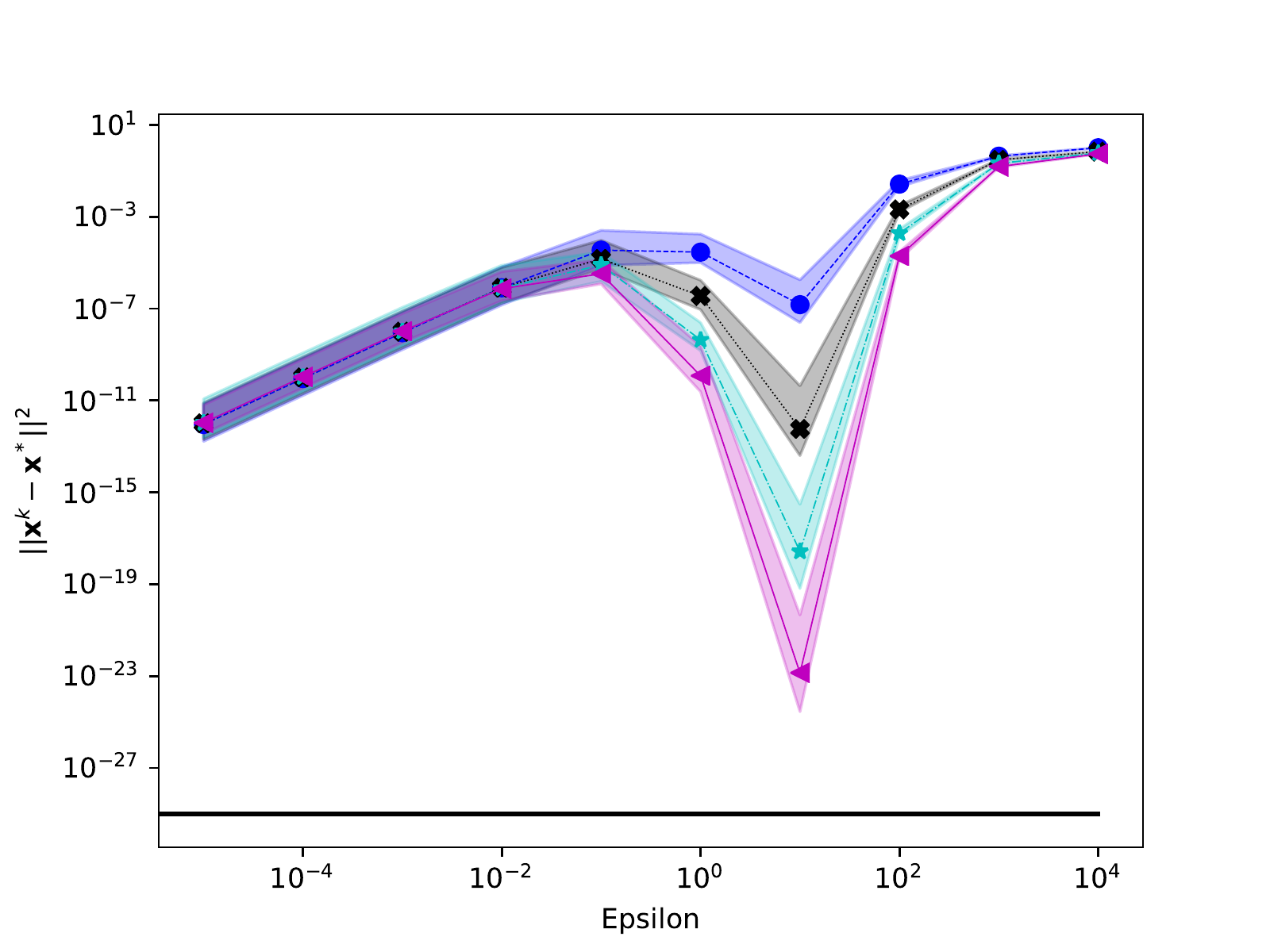}
    \end{subfigure}
    \newline
    \begin{subfigure}{.48\textwidth} 
        \includegraphics[width=\linewidth]{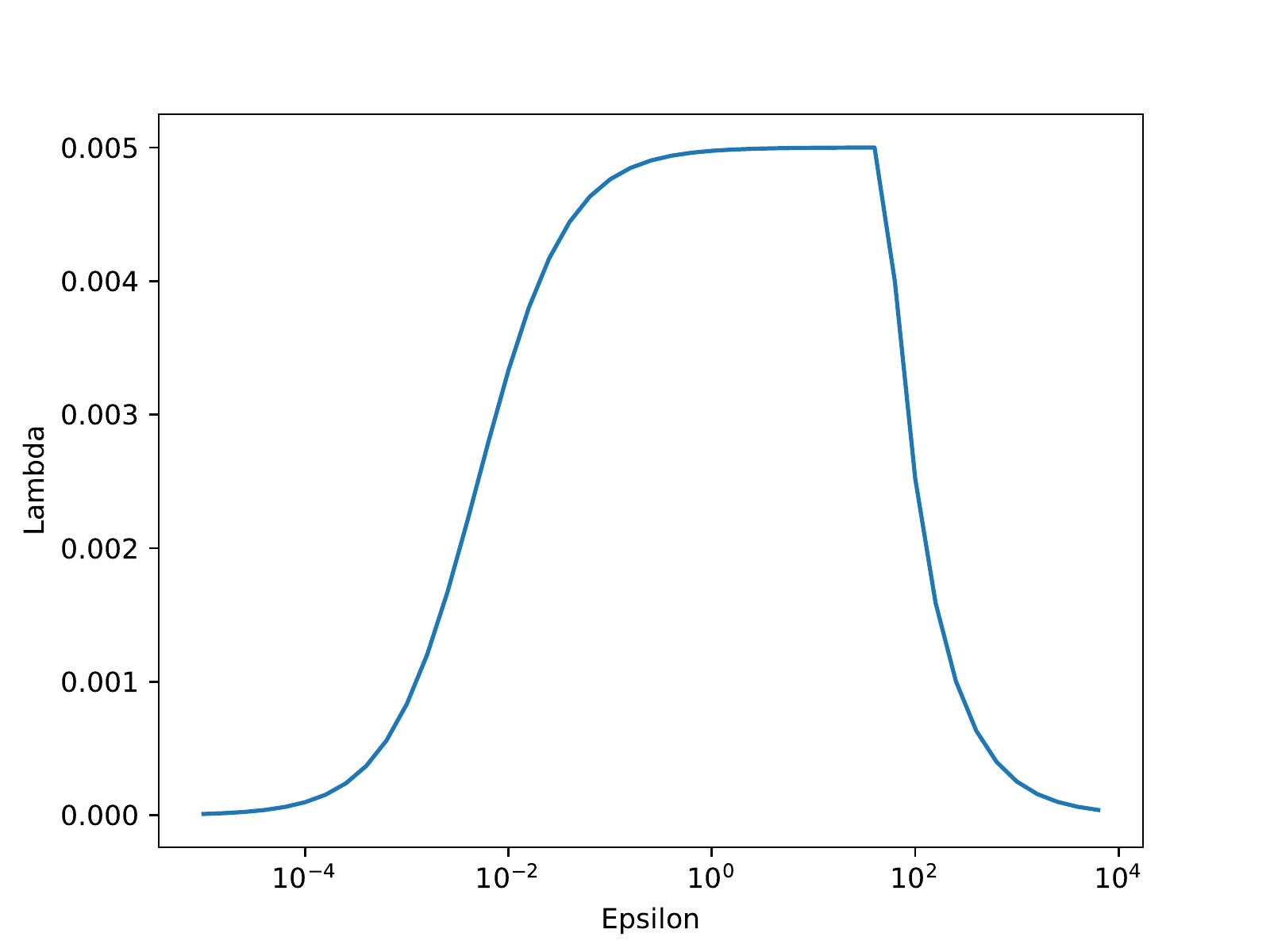}
        \caption{$200 \times 10$ Gaussian matrix.}
    \end{subfigure}
    \begin{subfigure}{.48\textwidth} 
        \includegraphics[width=\linewidth]{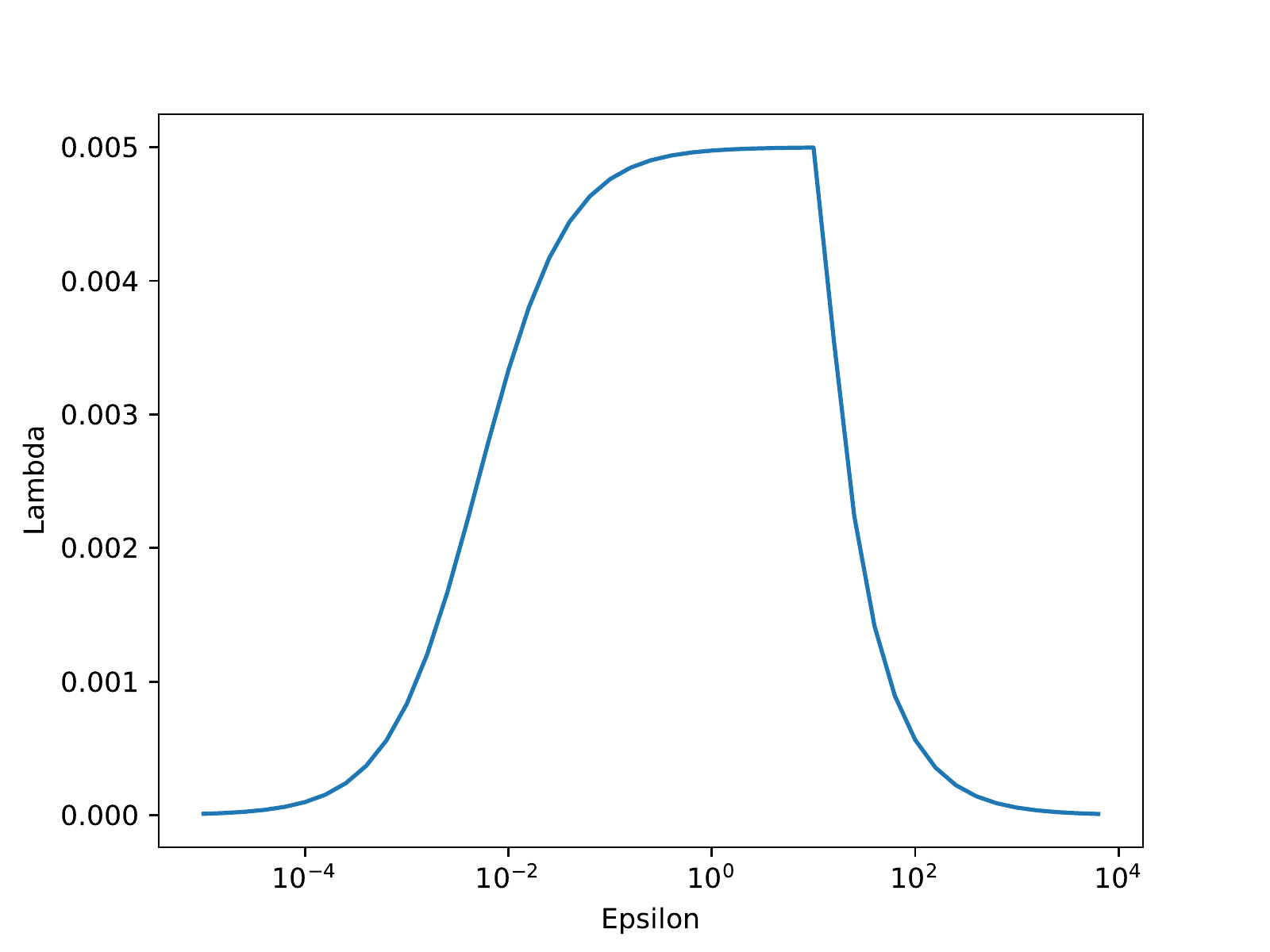}
        \caption{$200 \times 10$ coherent matrix.}
    \end{subfigure}
  \caption{The upper row of figures plot $\| \xk - \xopt\|^2$ as a function of $\epsilon$ for SAP-REK$(\epsilon)$. The lower row is a plot of $\lambda^+_{\min}(\mW_\epsilon)$ as a function of $\epsilon$. Notice the dip in each curve in the upper figures at $\epsilon = 100$ (upper left) and $\epsilon = 10$ (upper right). This highlights that  SAP-REK$(\epsilon=100)$ for a Gaussian matrix and SAP-REK$(\epsilon=10)$ overcome the convergence plateau. These faster overall convergence anomalies is explained by the lower plots of $\lambda^+_{\min}(\mW_\epsilon)$ as a function of $\epsilon$ because the convergence rate is inversely proportional to $\lambda^+_{\min}(\mW_\epsilon)$. Note: the $\epsilon$ axis is labeled Epsilon and the $\lambda^+_{\min}(\mW_\epsilon)$ axis is labeled Lambda.}
  \label{fig:epsilon}
\end{figure}

We also experiment with matrices of dimensions $200 \times 20$ and $400 \times 10$ and observe a similar phenomenon where  the most convergent SAP-REK$(\epsilon)$ methods correspond to $\epsilon = 10$ for Gaussian matrices and $\epsilon = 100$ for coherent matrices. Our findings our summarized in \cref{tab:eps_exp}. 
\begin{table}[H]
    \centering
    \begin{tabular}{c c|c c}
         & &\multicolumn{2}{c}{Type} \\
         & & Gaussian matrix & coherent matrix \\
        \hline
        \multirow{3}{*}{ Dimensions}  & $200 \times 10$ & 10 & 100 \\ 
         & $200 \times 20$ & 10 & 100  \\
         & $400 \times 10$ & 10 & 100 \\
    \end{tabular}
    \caption{The same experiment as first row of plots in \cref{fig:epsilon} is preformed with Gaussian and coherent matrices of dimensions $200 \times 20$ and $400 \times 10$. $\epsilon$ is sampled at each order of magnitude in $[10^{-5},10^4]$. The $\epsilon$ value yielding the fastest convergence is displayed in the center of the table for each given matrix type and dimension.}\label{tab:eps_exp}
\end{table}


We do a final experiment where we run SAP-REK$(\epsilon)$ and REK for $k = 10,000$ iterations to see if any SAP-REK$(\epsilon)$ eventually break convergence plateau for small $\epsilon$. The results of this experiment are in plots \cref{fig:xs}, \cref{fig:zs} and \cref{fig:xs_zs}. Rather than just plot $\norm{\xk - \xopt}^2$ as a function of $k$, we choose to include plots of $\norm{\zk - \zopt}^2$ and $\norm{\zk - \zopt}^2 - \epsilon \norm{\xk - \xopt}^2$. We plot \cref{fig:xs_zs} since it corresponds with our convergence result for SAP-REK$(\epsilon)$ in \cref{eq:convergence_result}.  Each line in \cref{fig:xs_zs} corresponds to a different $\epsilon$ value, so the y-axis label changes for each line plotted since it is a function of $\epsilon$.

\begin{figure}[h!]
\centering
\begin{subfigure}{.5\textwidth}
  \centering
  \includegraphics[width=\linewidth]{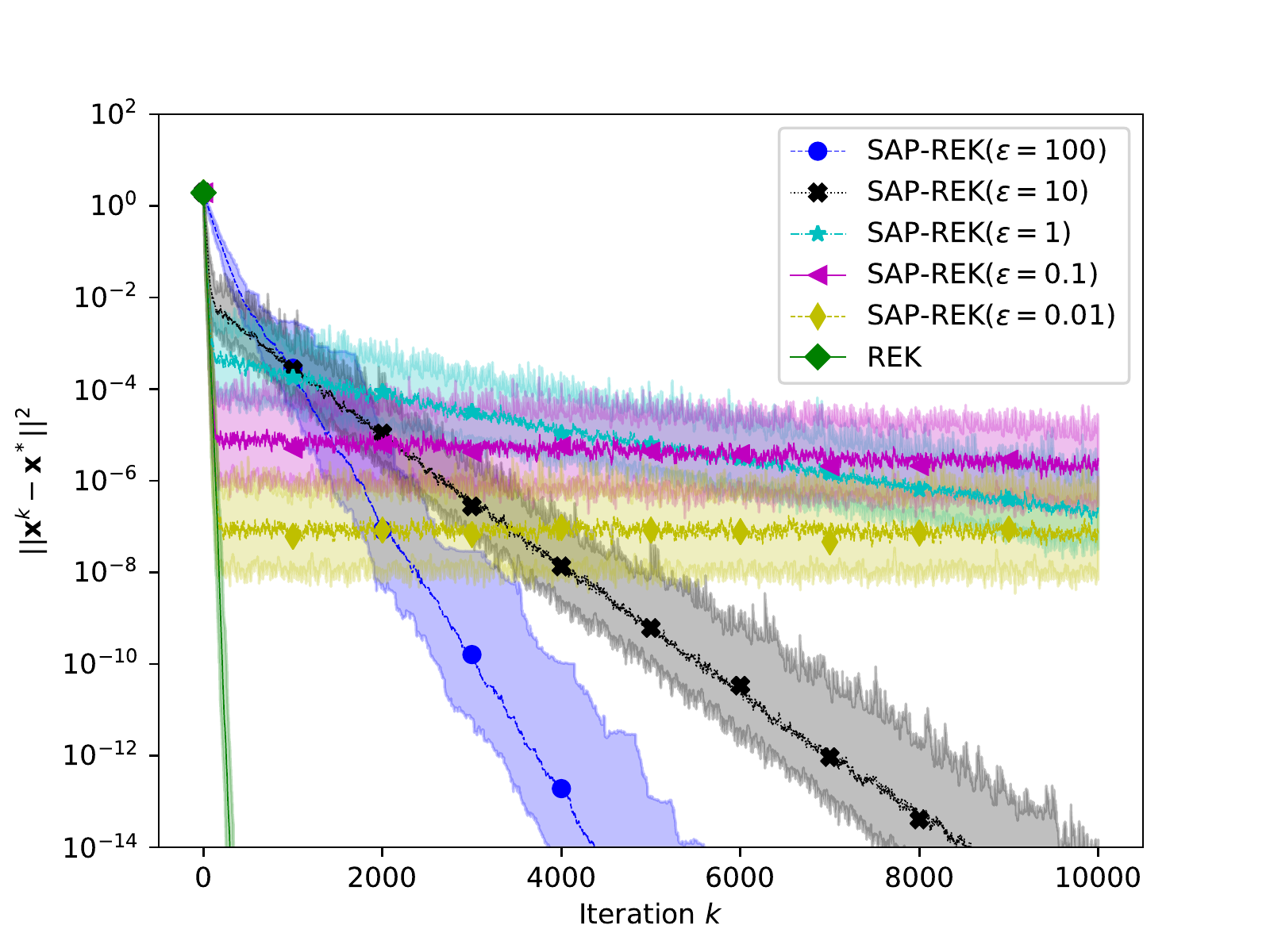}`
  \caption{Gaussian matrix}
  \label{fig:gauss_rek1}
\end{subfigure}%
\begin{subfigure}{.5\textwidth}
  \centering
  \includegraphics[width=\linewidth]{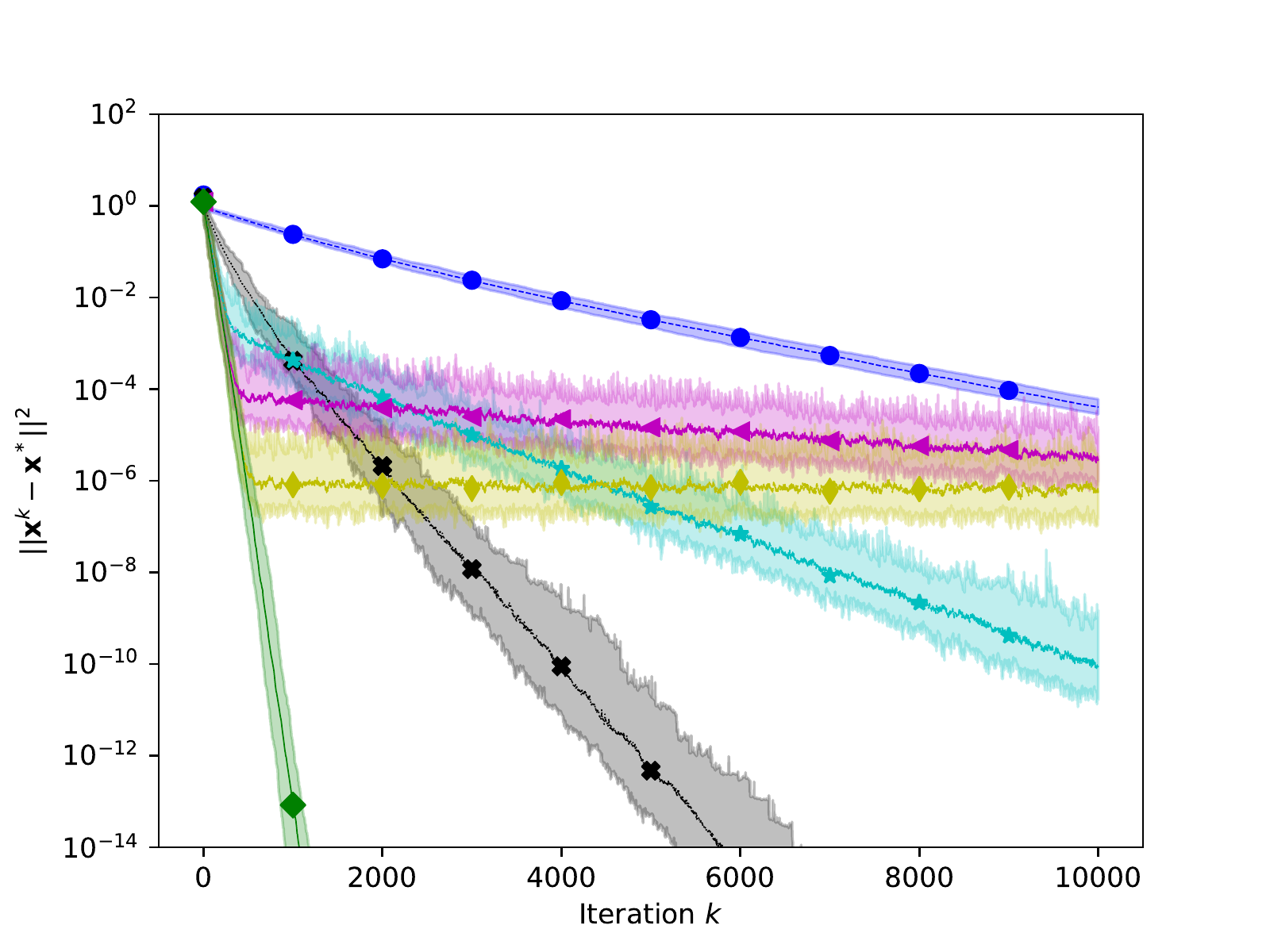}
  \caption{Coherent matrix}
  \label{fig:coherent_rek2}
\end{subfigure}
\caption{Plot of squared error norm vs. iteration. Matrices are $200 \times 10$. REK corresponds to $\epsilon = 0$.}
\label{fig:xs}
\end{figure}

The SAP-REK$(\epsilon)$ methods exhibit convergence behavior closer to REK for the first 700--800 iterations. SAP-REK$(\epsilon)$ with small $\epsilon$ appear hit an unbreakable plateau, whereas SAP-REK$(\epsilon)$ with a larger $\epsilon$ result in eventual convergence. These observations are consistent across \cref{fig:xs}, \cref{fig:zs} and \cref{fig:xs_zs}. In addition, we see that SAP-REK$(\epsilon = 100)$ and SAP-REK$(\epsilon = 10)$ are the most convergent SAP-REK for Gaussian and coherent matrices respectively. This aligns with the similar results in \cref{fig:epsilon}. In summary, these experiments suggest that, for most values of $\epsilon$, SAP-REK$(\epsilon)$ hit a convergence plateau for tall Gaussian and coherent matrices. However, there is a special window of $\epsilon$ values (near $\epsilon = 10$ for Gaussian matrices and $\epsilon = 100$ for coherent matrices) where SAP-REK$(\epsilon)$ breaks thee convergence plateau and eventually converges.




\begin{figure}[H]
\centering
\begin{subfigure}{.5\textwidth}
  \centering
  \includegraphics[width=\linewidth]{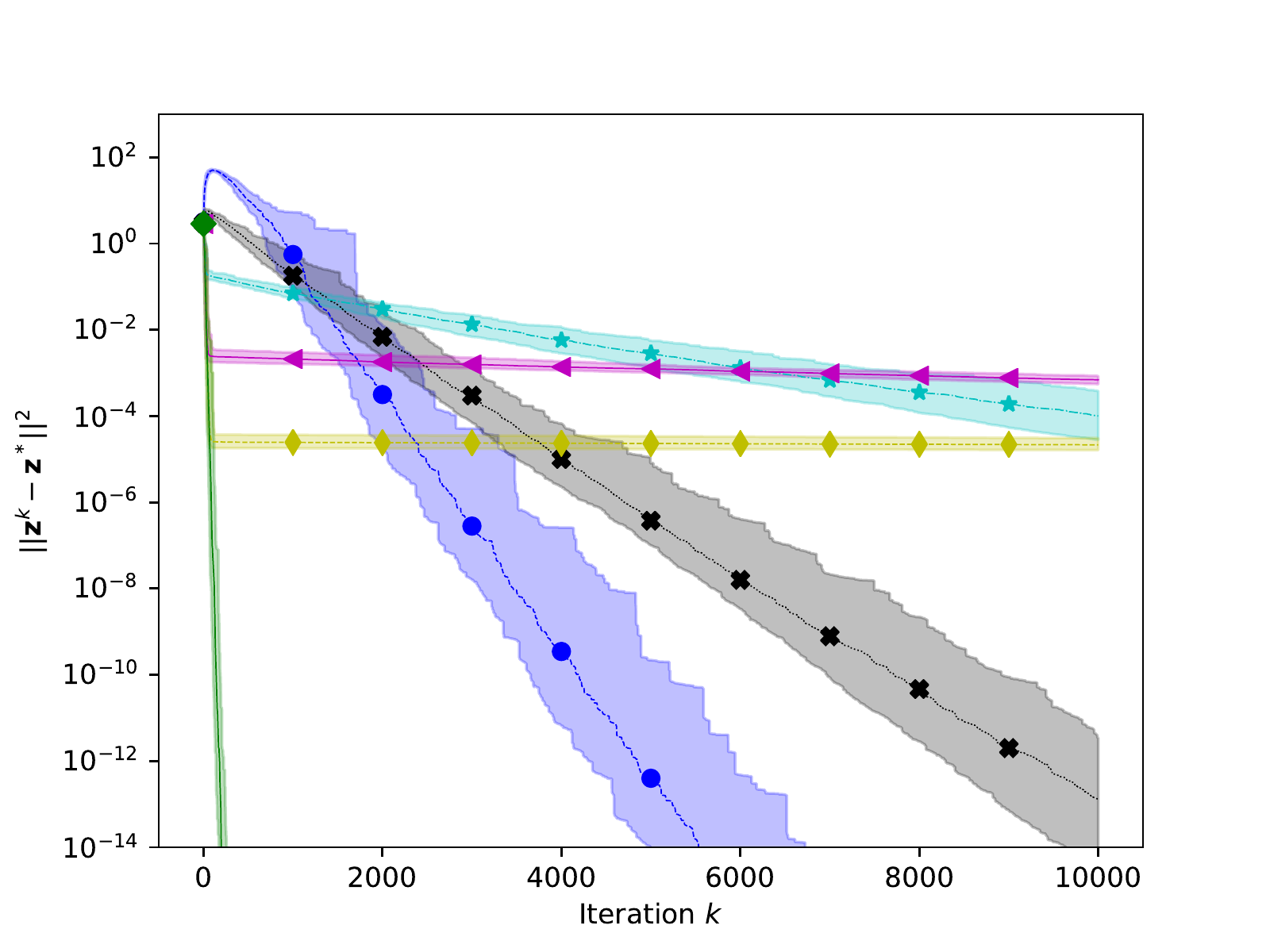}
  \caption{Gaussian matrix}
  \label{fig:gauss_rek_zs1}
\end{subfigure}%
\begin{subfigure}{.5\textwidth}
  \centering
  \includegraphics[width=\linewidth]{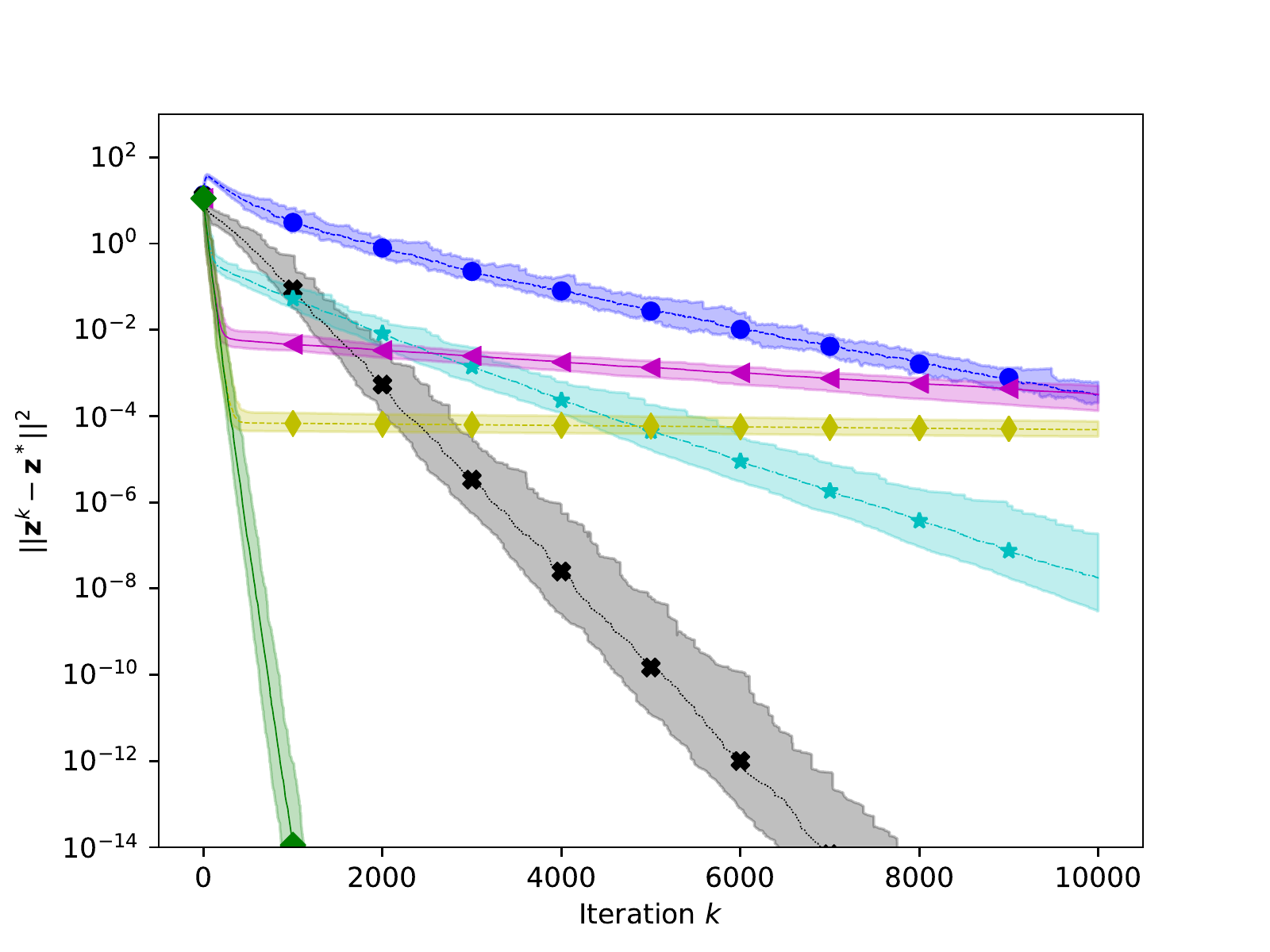}
  \caption{Coherent matrix}
  \label{fig:coherent_rek_zs1}
\end{subfigure}
\caption{Plot of the squared 2-norm between iterates of $z$ and $\zopt$ vs. iteration. Matrices are $200 \times 10$. REK corresponds to $\epsilon = 0$.}
\end{figure}
\begin{figure}[H]
\centering
\label{fig:zs}
\begin{subfigure}{.5\textwidth}
  \centering
  \includegraphics[width=\linewidth]{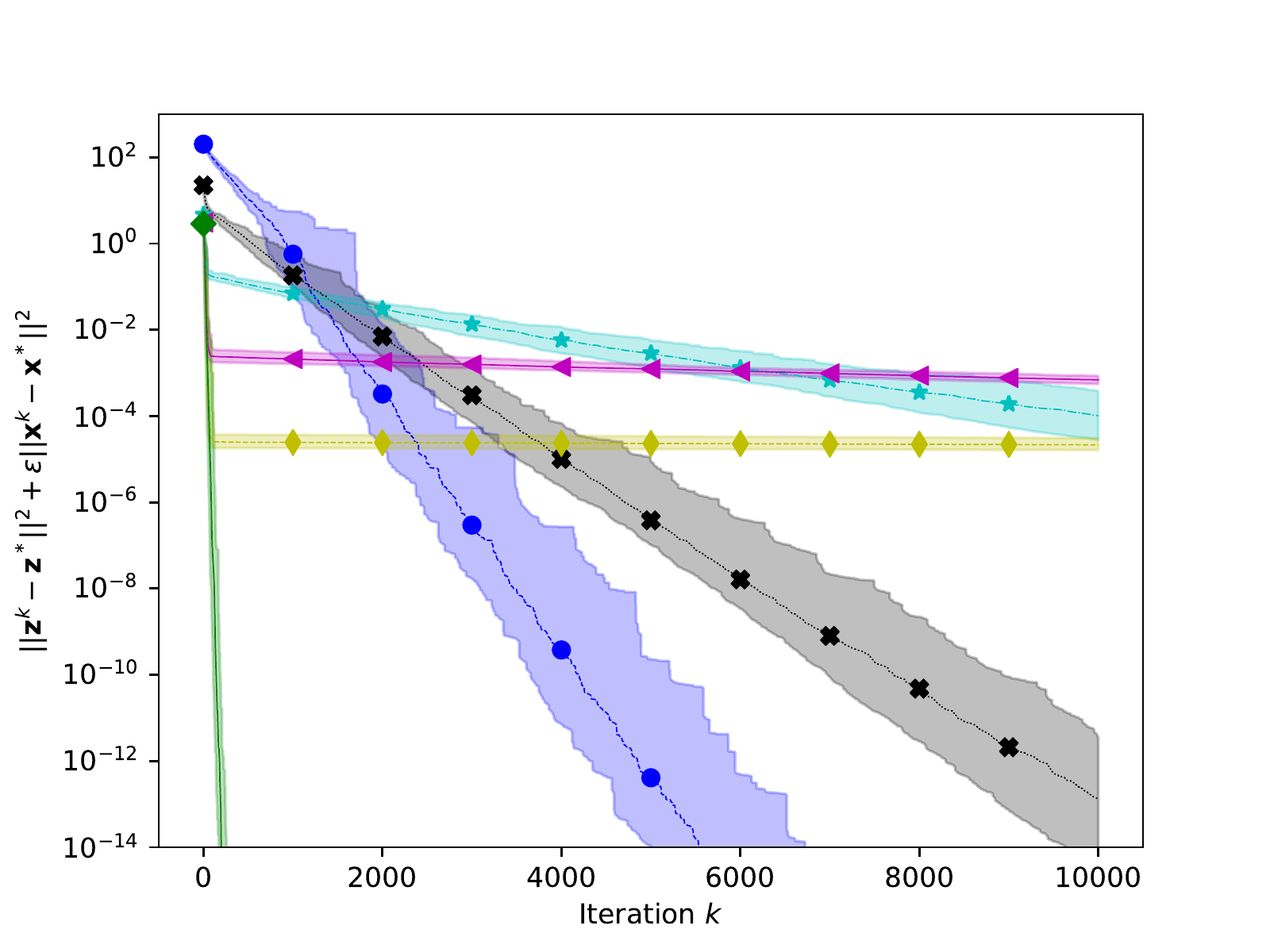}
  \caption{Gaussian matrix}
  \label{fig:gauss_rek_zs2}
\end{subfigure}%
\begin{subfigure}{.5\textwidth}
  \centering
  \includegraphics[width=\linewidth]{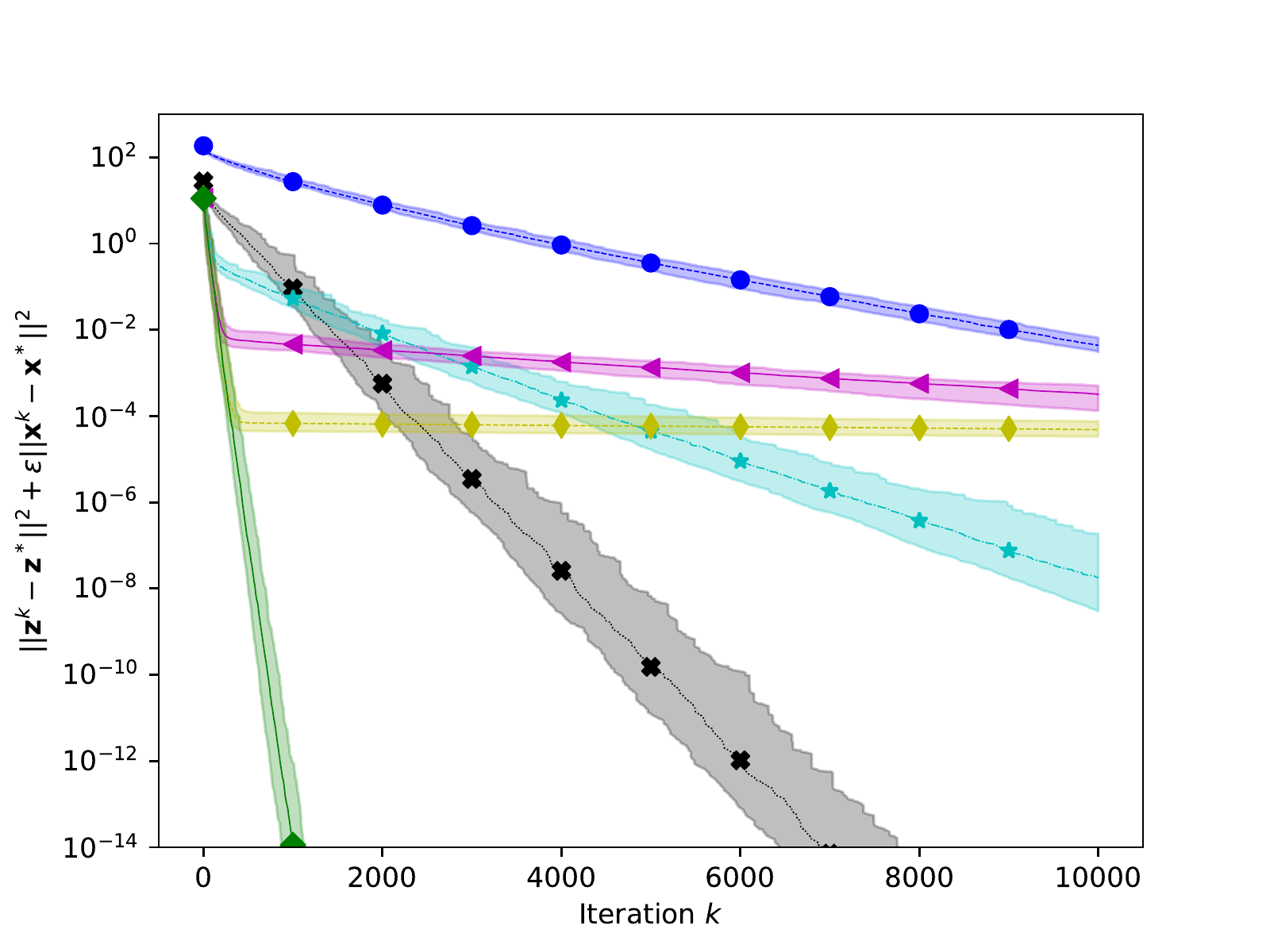}
  \caption{Coherent matrix}
  \label{fig:coherent_rek_zs2}
\end{subfigure}
\caption{Plot of $\norm{\zk - \zopt}^2 + \epsilon \norm{\xk - \xopt}^2$. Matrices are $200 \times 10$. REK corresponds to $\epsilon = 0$. For REK we plot $\norm{\zk - \zopt}^2$ since $\epsilon$ is not a parameter for REK. Note: the y-axis value in this plot is dependent on $\epsilon$, so we are really plotting a different value for each choice of $\epsilon$. }
\label{fig:xs_zs}
\end{figure}

\section{Conclusion and Future Work}
\label{conclusion}

In this work, we show that the widely used randomized extended Kaczmarz method for finding the least squares solutions of inconsistent systems is likely not a member of the Sketch-and-Project family of methods applied to a suitably embedded problem, which have previously been shown to include the original randomized Kaczmarz method and many of its variations. This presents a gap in the theory, which we bridge by showing that REK is in fact a limit point, in a certain sense, of the Sketch-and-Project family. We do this by constructing a new family of methods, SAP-REK$(\epsilon)$, within SAP, and show that taking the limit of the primary parameter $\epsilon\to 0$ recovers REK. We furthermore perform a detailed analysis of this family, showing that SAP-REK$(\epsilon)$ methods enjoy a similar flavor of linear convergence as REK, RK, and variants thereof. We demonstrate our methods experimentally, and show that they are, as expected, outperformed by REK but do exhibit limiting behavior toward it.

We believe that other extended methods such as randomized extended Gauss-Seidel and randomized extended block Kaczmarz may also be shown as limit points of SAP, and that there could be some unified 'Extended Sketch-and-Project' theory to be developed covering methods created both row and column sketches at each iteration. Furthermore, in our work we have shown that our limit does not recover the known convergence result for REK, suggesting that there is more work to be done to bridge this gap fully.

\section{Acknowledgements}

JDM was supported by NSF DGE-1829071. BJ was supported by NSF BIGDATA \#1740325 and NSF DMS \#2011140. NM was partially supported by NSF ATD \#1830676. NM thanks his advisor, Michael Kirby. BJ and JDM thank Deanna Needell for her advisorship and helpful conversations.




\bibliographystyle{abbrvnat}
\bibliography{main}

\end{document}